\documentclass{article}
\usepackage{amsmath,amsfonts,amsthm,amssymb,commath,mathtools}
\usepackage{comment}
\usepackage{tikz}
\usetikzlibrary{calc,angles,quotes,shapes.misc}
\usepackage[margin=1in]{geometry}
\usepackage{hyperref}

\tikzset{cross/.style={cross out, draw, 
		minimum size=2*(#1-\pgflinewidth), 
		inner sep=0pt, outer sep=0pt}}

\title{Quasi-Random Influences of Boolean Functions}
\author{Fan Chung and Nicholas Sieger}

\theoremstyle{plain}
\newtheorem{prop}{Proposition}[section]
\newtheorem{thm}[prop]{Theorem}
\newtheorem{conj}[prop]{Conjecture}
\newtheorem{lem}[prop]{Lemma}

\newtheorem{property}{Property}

\theoremstyle{definition}
\newtheorem{defn}[prop]{Definition}
\newtheorem{clm}[prop]{Claim}

\newtheorem{ex}[prop]{Example}

\theoremstyle{remark}
\newtheorem{rmk}[prop]{Remark}

\newtheorem{ques}[prop]{Question}

\newcommand{\N}{\mathbb{N}}
\newcommand{\Z}{\mathbb{Z}}

\newcommand{\R}{\mathbb{R}}
\newcommand{\C}{\mathbb{C}}
\newcommand{\F}{\mathbb{F}}

\newcommand{\Zp}[1]{\Z/#1\Z}

\DeclareMathOperator{\diam}{diam}

\newcommand{\half}{\frac{1}{2}}
\newcommand{\bone}{\mathbf{1}}

\DeclarePairedDelimiter{\lrceil}{\lceil}{\rceil}

\DeclarePairedDelimiter{\iprod}{\langle}{\rangle}

\DeclareMathOperator{\rank}{rank}

\DeclareMathOperator{\Range}{Range}

\DeclareMathOperator{\exv}{\mathbb{E}}
\DeclareMathOperator{\prob}{\mathbb{P}}

\newcommand{\gownorm}[2]{\Vert #1\Vert_{U(#2)}}
\newcommand{\fourier}[2]{\widehat{#1}\left(#2\right)}

\DeclareMathOperator{\emb}{em}

\DeclareMathOperator{\bhom}{bhom}
\DeclareMathOperator{\Subdiv}{Sub}

\newcommand{\spm}{\{1,-1\}}
\DeclareMathOperator{\Inf}{I}

\begin{document}
	\maketitle
	\begin{abstract}
		 We examine a hierarchy of equivalence classes of quasi-random properties of Boolean Functions. In particular, we prove an equivalence between a number of properties including balanced influences, spectral discrepancy, local strong regularity,  homomorphism enumerations of colored or weighted graphs and hypergraphs associated with Boolean functions as well as the $k$th-order strict avalanche criterion amongst others. We further construct  families of quasi-random boolean functions which exhibit the properties of our equivalence theorem and separate the levels of our hierarchy. 
	\end{abstract}
	
	\section{Introduction}
		We consider {\it Boolean Functions} that map binary strings of length $n$ to $\{\textit{True},\textit{False}\}$.  As boolean functions can encode a wide variety of mathematical and computational objects, such as decision problems, error-correcting codes, communication and cryptographic protocols, voting rules, propositional formulas, and arbitrary subsets of the set of binary strings, these functions are extremely well-studied in computer science, coding theory, cryptography, and data sciences to name a few areas of applications. For each application, many researchers have developed tools and perspectives unique to each area to study these boolean functions and have isolated key properties of boolean functions, for instance as the sensitivity of the function to changes in each coordinate, the size of its Fourier coefficients, or the distance of its support viewed as a binary code.
		
		The goal of this paper is to organize some of these properties of boolean functions into a hierarchy of quasi-random equivalence classes in the same style as the quasi-random equivalence for graphs proven by Chung, Graham, and Wilson in \cite{Chung1989} (for details, see section \secref{sec main result}). In our main theorem, we show how a number of known analytic properties of boolean functions, such as the $k$-th order strict avalanche criterion, restrictions of the function having small Fourier coefficients, and discrepancy of the Fourier coefficients, can be either strengthened or weakened so as to become equivalent to one another in the same sense as in \cite{Chung1989}. Motivated by the enumeration of ``sub-patterns'' within a larger object, we further show that several combinatorial properties of graphs built from our boolean function are equivalent to these analytic properties. These combinatorial properties include counts of rainbow embeddings of graphs and a co-degree condition on graphs defined from the boolean function. Finally, we give an explicit construction of a family of boolean functions which exhibits the properties in our main theorem. As it turns out, our construction depends crucially on the existence of good binary codes. 
		
		Our work continues the study of quasi-randomness of graphs and hypergraphs in the work of Chung, Graham, and Wilson \cite{Chung1989}.  Quasi-randomness theorems exist for other combinatorial objects, including Griffiths' results on oriented graphs \cite{Griffiths2011}, Cooper's work on permutations \cite{Cooper2004}, Chung and Graham's work on tournaments \cite{Chung1991b} and subsets of $\Z/N\Z$ \cite{Chung1992}, $k$-uniform linear Hypergraphs in works of Freidman and Widgerson \cite{friedman1989second} along with R{\"o}dl, Schacht, and Kohawakaya \cite{Kohayakawa2010} and finally Lenz and Mubayi \cite{Lenz2015}, and $k$-uniform general hypergraphs in papers of Chung and Graham \cite{Chung1990,Chung1990a,Chung1991a,Chung1992a}, Frankl, R{\"o}dl, Schacht, Kohawakaya, and Nagle in \cite{frankl1992uniformity,rodl2007counting,rodl2007regularity,Nagle2006,Kohayakawa2003}, surveyed  in the papers of Gowers \cite{Gowers2006,Gowers2007}. There are also several extant theories of quasi-randomness for boolean functions, implicitly in Chung and Graham's work on subsets of $\Z/N\Z$ \cite{Chung1992} and explicitly in O'Donnell's textbook \cite{o'donnell_2014}, Castro-Silva's monograph \cite{castro-silva2021quasirandomness}, and Chung and Tetali's work on communication complexity \cite{chung1993communication}. We shall later prove that our theory of quasi-random boolean functions is distinct from each of these theories and is stronger than several of these theories and incomparable with the others. 
		
		Our paper is organized as follows. In section \secref{sec defns}, we give the necessary preliminaries to state our quasi-random properties. Then we define the quasi-random properties and present our main results in section \secref{sec main result}. We define the properties in extant theories of quasi-randomness for boolean functions in section \secref{sec previous theories}.  The main results of this paper are summarized in two flowcharts as seen in Figures \ref{figure diagram of main proof} and \ref{figure relations between different theories}. The proof of our main equivalence theorem appears in section \secref{sec main proof}, followed by the construction of our quasi-random functions in section \secref{sec constructions}. We then prove our comparison theorems which place our theory of quasi-randomness relative to several extant theories in section \secref{sec comparisons}. Finally, we conclude with some remarks and open problems in section \secref{sec conclusion}.

	\section{Preliminaries}\label{sec defns}
		 We refer the reader to O'Donnell's book \cite{o'donnell_2014} for any undefined terminology. We identify the set of binary strings $\{0,1\}^n$ with elements of $\F_2^n$ via a choice of basis for $\F_2^n$, and then define a \textit{boolean function} to be a map $f:\F_2^n\to \spm$. We then equip the space of all maps $g:\F_2^n\to \C$ with the following inner product:
		\[
			\iprod{f,g} := \exv_{x\in\F_2^n} f(x)\overline{g(x)} = \frac{1}{2^n} \sum_{x\in \F_2^n} f(x)\overline{g(x)}
		\] where $\overline{z}$ denotes complex conjugation of $z\in \C$. For each $\gamma\in \F_2^n$, the \textit{Fourier character} $\chi_{\gamma}:\F_2^n\to \spm$ is 
		\[
			\chi_{\gamma}(x) := (-1)^{\gamma\cdot x}
		\] where $\gamma\cdot x := \sum_{i = 1}^{n} \gamma_ix_i$ is the usual dot product. The Fourier characters are an orthonormal basis for the space of all maps $g:\F_2^n\to \C$ with the inner product as defined above. Therefore, every function $g:\F_2^n\to \C$ has unique \textit{Fourier coefficients} $\fourier{g}{\gamma}$ where $\fourier{f}{\gamma} = \iprod{g,\chi_{\gamma}}$. The \textit{convolution} of two functions $g\text{ and }h:\F_2^n\to \C$ is 
		\[
			(g*h)(x) := \exv_{y\in\F_2^n} f(x + y)\overline{g(y)}.
		\] Note that $\fourier{g*h}{\gamma} = \fourier{g}{\gamma}\fourier{h}{\gamma}$.

		We will also need to track the size of individual vectors in $\F_2^n$. For a vector $x\in \F_2^n$, its  \textit{Hamming weight}, denoted $|x|$, is $\abs{\{i\in [n]: x_i = 1\}}$. Similarly, the \textit{Hamming distance} between two vectors $x\text{ and }y\in \F_2^n$ is $\abs{x - y}$. For a subset $S\subseteq \F_2^n$, its \textit{diameter} is $\diam(S) := \max_{x,y\in S} \abs{x - y}$. The \textit{Hamming ball} of radius $d$ in $\F_2^n$ and centered at the vector $x\in \F_2^n$, denoted by $B_d(n,x)$, is $\{y\in \F_2^n: \abs{x - y} \leq d\}$. 
		
		For a proposition $P(x)$, let $[P(x)] := \begin{cases}
		1 & P(x)\\
		0 & \neg P(x)
		\end{cases}$ denote the \textit{indicator function} for $P(x)$. We will write $0$ for the zero vector in $\F_2^n$ throughout, and write $\bone\in \F_2^n$ for the all-ones vector. If $\mathcal{S}$ is a distribution on a set $S$, and $P(x)$ is a proposition on the variable $x\in S$, then
		\[
			\prob_{x\sim \mathcal{S}}\left[P(x)\right]
		\] will denote the probability that $P(x)$ holds where $x$ is drawn from the distribution $\mathcal{S}$. 
		 Whenever we write the expectation or probability over a set, such as $\exv_{x\in \F_2^n}$, the expectation or probability is taken with respect to the uniform distribution.
		 
		\smallskip
		Here we state the definitions concerning various aspects of Boolean functions that will be used later.
	\subsection{The influences of Boolean functions}	
	The notion of ``influences''  is prominent in both analysis of boolean functions and in cryptography. 
			\begin{defn}\label{defn directional influence}
				For $\gamma\in \F_2^n$, the \textit{$\gamma$-Influence} of $f$ is
				\[
				\Inf_{\gamma}[f] := \prob_{x\in\F_2^n}\left[f(x) \neq f(x + \gamma)\right].
				\]
			\end{defn}
			Note that $\Inf_{0}[f]$ is always $0$. Furthermore, for $\gamma\in \F_2^n$ with $\gamma_i = 1$ and $\gamma_j = 0$ for $j \neq i$, $\Inf_{\gamma}[f]$ is precisely the influence of coordinate $i$ as studied extensively in O'Donnell \cite{o'donnell_2014}.
	\subsection{The spectral sampling of Boolean functions}
	Parseval's Theorem states that for $f:\F_2^n\to \spm$ 
			\[
				\sum_{\gamma\in\F_2^n} \fourier{f}{\gamma}^2 = \exv_{x\in\F_2^n} \left[f(x)^2 \right] = 1.
			\] Thus the Fourier coefficients of $f$ define a probability distribution on $\F_2^n$ as follows:

				\begin{defn}\label{defn spectral sample}
					For a fixed boolean function $f:\F_2^n\to\spm$, the \textit{Spectral Sample} $\mathcal{S}_f$ is the distribution on $\F_2^n$ where
					\[
					\prob_{\gamma\sim \mathcal{S}_f}\left[\gamma = \delta\right] = \fourier{f}{\delta}^2
					\] for a fixed $\delta\in\F_2^n$.
				\end{defn}

		\subsection{Subcubes and the counting of subcubes}
		Let $[n]$ denote the set $\{1,\dots,n\}$, and for $S\subseteq [n]$, let $\overline{S}$ denote $[n]\setminus S$. Given a set $S\subseteq [n]$, and two vectors $x\in \F_2^S$, $y\in \F_2^{\overline{S}}$, let $x\underset{S}{\otimes} y$ denote the vector where
		\[
			(x\underset{S}{\otimes} y)_i = \begin{cases}
			x_i & i \in S\\
			y_i & i\in \overline{S}
			\end{cases}.
		\] The \textit{Subcube} defined by a set $S\subseteq [n]$ and a vector $z\in \F_2^{\overline{S}}$ is the set
		\[
			C(S,z) := \{x\underset{S}{\otimes} z: x\in \F_2^S\}.
		\] We say that the \textit{dimension} of a subcube $C(S,z)$ is $\abs{S}$. Note that $C([n],\eta)$ where $\eta$ is the empty string is precisely the hypercube $Q_n$. In Figure \ref{figure subcube examples}, we have two examples of subcubes. We are also concerned about boolean functions restricted to a subcube:
		\begin{defn}\label{defn restrictions}
			The \textit{restriction} of $f:\F_2^n\to \spm $ to the subcube $C(S,z)$ is the boolean function $f|_{S,z}:\F_2^S\to \spm$ defined by
			\[
			f|_{S,z}(x) = f(x\underset{S}{\otimes} z)
			\]
		\end{defn} 
		If $S = \emptyset$, then $f|_{S,z}(x)$ is the constant function $f(z)$, and if $S = [n]$, then we recover $f$ itself.  
		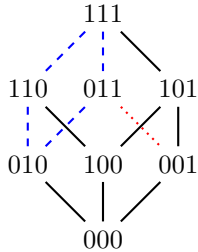
\begin{figure}[htbp]
			\centering
			\begin{tikzpicture}
				\node (000) at (0,0) {$000$};
				
				\node (001) at (1,1) {$001$};
				\node (010) at (-1,1) {$010$};
				\node (100) at (0,1) {$100$};
				
				\node (101) at (1,2) {$101$};
				\node (011) at (0,2) {$011$};
				\node (110) at (-1,2) {$110$};
				
				\node (111) at (0,3) {$111$};
				
				\draw[thick] (000) -- (010);
				\draw[thick] (000) --  (100);
				\draw[thick] (000) -- (001);

				\draw[thick] (001) -- (101);
				\draw[thick] (101) -- (100);
				\draw[thick] (100) -- (110);
				\draw[blue, dashed,thick] (110) -- (010);
				\draw[blue, dashed,thick] (010) -- (011);
				\draw[red,dotted,thick] (011) -- (001);

				\draw[thick] (111) -- (101);
				\draw[blue, dashed,thick] (111) -- (011);
				\draw[blue, dashed,thick] (111) -- (110);
			\end{tikzpicture}
			\caption{The blue dashed lines in the figure indicate the $2$-dimensional subcube $C(\{1,3\},1)$, i.e. the set of all length $3$ binary strings with a $1$ in the second coordinate. The red dotted line indicates the $1$-dimensional subcube $C(\{2\},01)$.}\label{figure subcube examples}
		\end{figure}


\subsection{Combinatorial aspects of Boolean functions}	
	Here we give several useful combinatorial interpretations of Boolean functions that are of interest of their own right. For two sets $A,B$, let $A\hookrightarrow B$ denote the set of all injective functions from $A$ to $B$. Let $A\sqcup B$ denote the disjoint union of the sets $A$ and $B$. 
\subsubsection{Bipartite Cayley Graphs}	
			Given a bipartite graph $G = (U\sqcup V,E)$, we say $U$ is its \textit{left part}, denoted $L(G)$, and $V$ its \textit{right part}, denoted $R(G)$.  For $v\in V(G)$, let $N_G(v)$ denote the neighborhood of $v$ in $G$. 
			\begin{defn}\label{defn bipartite Cayley graph}
				For a boolean function $f:\F_2^n\to \spm$, the \textit{bipartite Cayley graph} of $f$, denoted $BC(f)$, is the bipartite graph on vertex set $\F_2^n \sqcup \F_2^n$ where $u\sim v \iff f(u + v) = -1$.
			\end{defn}

			We observe that the number of edges in $BC(f)$ is exactly $2^{2n}\prob_{x\in\F_2^n}[f(x) = -1]$. The bipartite Cayley graph is associated to the Cayley graph with vertex set $\F_2^n$ and edge set generated by $\{x\in \F_2^n: f(x) = -1\}$. See Figure \ref{figure double cover example} for a more explicit example of $BC(f)$. 
			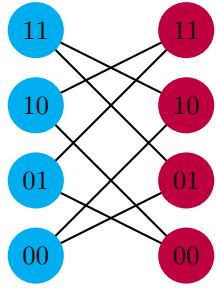
\begin{figure}[htp]
				\centering
				\begin{tikzpicture}
				\coordinate (00L) at (-1,-1);
				\coordinate (01L) at (-1,0);
				\coordinate (10L) at (-1,1);
				\coordinate (11L) at (-1,2);
				
				\coordinate (00R) at (1,-1);
				\coordinate (01R) at (1,0);
				\coordinate (10R) at (1,1);
				\coordinate (11R) at (1,2);
				
				\draw[thick] (00L) -- (10R);
				\draw[thick] (00L) -- (01R);
				\draw[thick] (10L) -- (00R);
				\draw[thick] (10L) -- (11R);
				\draw[thick] (01L) -- (00R);
				\draw[thick] (01L) -- (11R);
				\draw[thick] (11L) -- (01R);
				\draw[thick] (11L) -- (10R);
				
				\node[circle,fill=cyan] at (00L) {$00$};
				\node[circle,fill=cyan] at (01L) {$01$};
				\node[circle,fill=cyan] at (10L) {$10$};
				\node[circle,fill=cyan] at (11L) {$11$};
				\node[circle,fill=purple] at (00R) {$00$};
				\node[circle,fill=purple] at (10R) {$10$};
				\node[circle,fill=purple] at (01R) {$01$};
				\node[circle,fill=purple] at (11R) {$11$};
				\end{tikzpicture} 
				\caption{The bipartite Cayley graph, $BC(g)$, of the function $g(z) = (-1)^{z_1 + z_2}$ for $z_1,z_2\in \F_2$. $g$ encodes the XOR function on $2$ bits.}
				\label{figure double cover example}
			\end{figure}
\subsubsection{Embeddings and homomorphisms of bipartite subgraphs}\label{sssection bipartite graph homs}	

			We consider graph homomorphisms between bipartite graphs which preserve the bipartition. 
			\begin{defn}\label{defn bipartite graph homomorphism}
				A \textit{Bipartite Graph Homomorphism} from $H = (U\sqcup V,F)$ to $G = (A\sqcup B,E)$ is a pair of injective maps $(\psi,\phi)$ where $\psi:U\to A$, $\phi:V\to B$, and
				\[
				(u,v)\in F \implies (\psi(u),\phi(v)) \in E.
				\]
			\end{defn}
			Note that we explicitly define our homomorphisms to be injective. The set of all bipartite graph homomorphisms of a fixed bipartite graph $H$ into another bipartite graph $G$ will be denoted by $BHOM(H,G)$. Let
			\[
				\bhom(H,G) := \exv_{\psi:U \hookrightarrow A} \exv_{\phi:V \hookrightarrow B} \prod_{(u,v)\in F} [(\psi(u),\phi(v)) \in E]
			\] denote the normalized size of $BHOM(H,G)$.
			We also consider bipartite graph homomorphisms in which the left part is fixed by a particular injection $\psi$. The set of all bipartite graph homomorphisms with a fixed left map $\psi$ will be denoted by $BHOM_\psi(H,G)$. Let
				\[
			\bhom_\psi(H,G) := \exv_{\phi:V \hookrightarrow B}  \prod_{(u,v)\in F} [(\phi(u),\psi(v)) \in E]
			\] denote the normalized size of $BHOM_\psi(H,G)$.
			 We say that an injection $\psi$ has \textit{diameter} at most $k$ if the image of $\psi$ is a set of diameter at most $k$. 
	
\subsubsection{Colored multigraphs}
	The following definition is inspired by the work of Aharoni et al on rainbow extremal problems \cite{aharoni2018rainbow}.
	\begin{defn}\label{defn colored multigraph}
				An \textit{edge-colored multigraph} $M$ with color set $K$ is a multigraph with an edge-coloring using colors in $K$ such that multiple edges between the vertices $u$ and $v$ cannot have the same color.
			\end{defn}
			We will typically think of the edges of an edge-colored multigraph as a subset of $V\times V \times K$.

			\begin{defn}\label{defn rainbow Hamming graph}
				For fixed $f:\F_2^n\to\spm$, $k \geq 1$, and $K\subseteq \F_2^n$, the \textit{rainbow Hamming graph} $RHG(k,f)$ is the colored multigraph on the vertex set $B_k(n,0)$ with color set $K = \F_2^n$ and edge set defined as
				\[
				\{(u,v,x)\in V\times V\times K: f(u + x) = f(v + x)\}.
				\]
			\end{defn}
			 For an explicit example of a rainbow Hamming graph, see Figure \ref{figure rainbow Hamming graph example}.
			\begin{figure}
				\centering
				\begin{tikzpicture}
				\node[circle,fill=cyan] (00) at (0,0) {$00$};
				\node[circle,fill=cyan] (01) at (2,2) {$01$};
				\node[circle,fill=cyan] (10) at (-2,2) {$10$};

				\draw[red,thick] (10) edge[bend left]  node[midway, above,text=black] {$00$} (01) ;
				\draw[blue,thick] (10) edge[bend right]  node[midway, below,text=black] {$01$} (00) ;
				\draw[green,thick] (01) edge node[midway, above,text=black] {$10$} (00) ;
				\draw[yellow,thick] (10) edge node[midway, above,text=black] {$11$} (01) ;
				\draw[yellow,thick] (00) edge[bend right] node[midway, above,text=black] {$11$} (01) ;
				\draw[yellow,thick] (10) edge node[midway, above,text=black] {$11$} (00) ;
				
				\end{tikzpicture} 
				\caption{The rainbow Hamming graph $RHG(1,h)$ of the function $h(z) = (-1)^{1 - z_1z_2}$ where $z_1,z_2\in \F_2$. Each edge is labeled by the string in $\F_2^2$ which defines its color. Note that $h$ encodes the OR function.}\label{figure rainbow Hamming graph example}
			\end{figure}
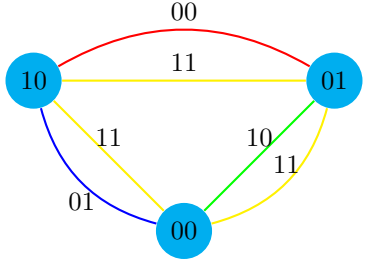

\subsubsection{Rainbow embeddings}			
			
			We consider graph homomorphisms into a colored multigraph which agree with the coloring.
			\begin{defn}\label{defn rainbow embedding}
				Let $M$ be a colored multigraph with color set $K$ and let $G$ be a fixed (simple) graph. A \textit{rainbow embedding} of $G$ into $M$ is a injective coloring $\chi:E(G) \to K$ and an injective map $\phi:V(G) \to V(M)$ such that 
				\[
				(u,v)\in E(G) \implies (\phi(u),\phi(v),\chi((u,v))\in E(M).
				\] 
			\end{defn}
			These embeddings are also considered in the work of Alon and Marshall \cite{alon1998hom}.

			For a fixed graph $G$, a fixed colored multigraph $M$ with color set $K$, let 
			\[
			\emb(G,M) :=  \exv_{\phi:V(G)\hookrightarrow V(M)} \exv_{\chi:E(G)\hookrightarrow K} \prod_{(u,v)\in E(G)} [(\phi(u),\phi(v),\chi((u,v))) \in E(M)]
			\] be the normalized count of rainbow embeddings of $G$ into $M$. If we additionally fix the injection $\phi:E(G)\hookrightarrow K$, let \[
			\emb_{\phi}(G,M) :=  \exv_{\chi:E(G)\hookrightarrow K} \prod_{(u,v)\in E(G)} [(\phi(u),\phi(v),\chi((u,v))) \in E(M)]
			\] be the normalized count of rainbow embeddings with a fixed map $\phi$. 
	
\subsection{Bent Functions}
	We consider a specific class of boolean functions originally defined by Rothaus \cite{Rothaus1976}.
	\begin{defn}\label{defn bent function}\cite{Rothaus1976}
		For $n$ even, a boolean function $f:\F_2^n\to \spm $ is \textit{bent} if
		\[
		\abs{\fourier{f}{\gamma}} = 2^{-n/2}
		\] for every $\gamma\in\F_2^n$. 
	\end{defn} 
	Note that bent functions only exist for $n$ even. 
	
	We will need the following lemma:
	\begin{prop}\cite{Rothaus1976}\label{lem convolution of bent functions}
		Let $g:\F_2^n\to \spm$ be bent. Then 
		\[
		(g*g)(x) = [x = 0].
		\]
	\end{prop}
	
	\begin{ex}\label{ex inner product fn}
		The \textbf{Inner Product} function $IP:\F_2^{2m}\to \spm$ is defined by
		\[
		IP(z) := (-1)^{\sum_{i = 1}^{m} z_iz_{m + i}} = (-1)^{z_1\cdot z_2}
		\] where $z_1$ is the first $m$ bits of $z$ and $z_2$ is the last $m$ bits of $z$.

		For the sake of completeness, we show that $IP$ is in fact a bent function. To calculate its Fourier coefficients, fix $\gamma\in \F_2^{2m}$, and let $\gamma_1,\gamma_2\in \F_2^m$ denote the first $m$ bits of $\gamma$ and the last $m$ bits respectively. For $x\in \F_2^{2m}$, define $x_1,x_2$ similarly. Then,
		\begin{align}
		\fourier{IP}{\gamma} &= \exv_{x\in \F_2^{2m}} IP(x)\chi_{\gamma}(x)\nonumber\\
		&= \exv_{x_1\in \F_2^m} \exv_{x_2\in \F_2^m} (-1)^{x_1\cdot x_2 + \gamma_1\cdot x_1 + \gamma_2 \cdot x_2}\nonumber\\
		&= \exv_{x_1\in \F_2^m} (-1)^{\gamma_1 \cdot x_1}\exv_{x_2\in \F_2^m} (-1)^{(x_1 + \gamma_2)\cdot x_2}\nonumber\\
		&= \exv_{x_1\in \F_2^m} (-1)^{\gamma_1 \cdot x_1}[x_1 = \gamma_2]\label{eqn IP proof orthogonality}\\
		&= (-1)^{\gamma_1\cdot \gamma_2}2^{-m}\nonumber
		\end{align} where we use the fact that Fourier characters are orthogonal in line (\ref{eqn IP proof orthogonality}). Thus $IP$ is a bent function. 
	\end{ex}

	\section{Quasi-random Properties and the Main Results}\label{sec main result}
		We describe a number of quasi-random properties of boolean functions. The properties involve two parameters, denoted by $d$ and $\epsilon$, where $\epsilon$ indicates the error bound and $d$ is often related to the rank or dimension of patterns or objects in the property. The proof of the equivalence of these properties is given in section \secref{sec main proof}.

			If $f:\F_2^n\to \spm$ is chosen uniformly at random, we expect that the $\gamma$-influence (see Definition \ref{defn directional influence}) should be close to $\half$. Our first quasirandom property formalizes that notion: 
			\begin{property}\label{property directional influences}
				A boolean function $f:\F_2^n\to \spm$ with $\abs{\fourier{f}{0}} <\half$ has the \textbf{Balanced Influences Property} $INF(k,\epsilon)$ if the $\gamma$-Influence of $f$ is close to $\half$ for every nonzero $\gamma$ in the Hamming ball of radius $k$ centered at $0$, i.e. 
				\[
				\abs{\Inf_{\gamma}[f] - \half} < \epsilon
				\] for every $\gamma \neq \vec{0}$ in the Hamming ball of radius $k$ at $\vec{0}$.
			\end{property} We remark that $INF(k,0)$ is also known as the $k$th-Order Avalanche Criterion as studied in cryptography \cite{forre1990spectral}.

			For $f:\F_2^n\to\spm$ drawn uniformly from all boolean functions, the expected spectral sample (see Definition \ref{defn spectral sample}) is $\frac{1}{2^n}$ on each vector in $\F_2^n$. In particular, the total weight of the uniform distribution on a subcube of dimension $n - k$ is exactly $2^{k - n}$.  Our next quasi-random property states that $\mathcal{S}_f$ is not far from the uniform distribution on on subcubes.
			
			\begin{property}\label{property spectral discrepancy}
				A boolean function $f:\F_2^n\to \spm$ with $\abs{\fourier{f}{0}} <\half$ has the \textbf{Spectral Discrepancy Property} $SD(k,\epsilon)$ if the spectral sample of $f$ has total weight close to $2^{l - n}$ on every subcube of dimension $l$ where $l \geq n - k$ i.e.
				\[
				\abs{\prob_{z\sim \mathcal{S}_f}[z\in H] - 2^{\dim(C(S,z)) - n}} < \epsilon
				\] for every subcube $H$ of dimension at least $n - k$. 
			\end{property}
			One can think of this property as a form of ``discrepancy'' for the spectral sample.
			
			Next, we have a counting property on subcubes via the notion of restricted functions (see  Definition \ref{defn restrictions}).
			 As $f|_{S,z}$ is a map $\F^d \to \spm$ for $\abs{S} = d$, we can consider its Fourier coefficients. The next quasi-random property states that these Fourier coefficients are quite small on average.
			
			\begin{property}\label{property restriction fourier}
				A boolean function $f:\F_2^n\to \spm$ with $\abs{\fourier{f}{0}} <\half$ has the \textbf{Restriction Fourier Property} $RF(k,\epsilon)$ if the average restriction of $f$ is nearly a bent function on any subcube of dimension at most $k$, i.e. 
				\[
				\abs{\exv_{z\in \F_2^{\overline{S}}}\left[ \fourier{f|_{S,z}}{\gamma}^2\right] - 2^{-\dim(C(S,z))}} < \epsilon
				\] for every subcube $C(S,z)$ of dimension at most $k$ and every $\gamma\in \F_2^S$.
			\end{property}		
		
			The next property states that we can control certain patterns in the restrictions of $f$. 
			\begin{property}\label{property restriction convolutions}
				A boolean function $f:\F_2^n\to \spm$ with $\abs{\fourier{f}{0}} <\half$ has the \textbf{Restriction Convolution Property} $RC(k,\epsilon)$ if the average self-convolution of restrictions of $f$ to subcubes of dimension at most $k$ is close to the indicator function of the $0$ vector, i.e. 
				\[
					\abs{\exv_{z\in \F_2^{\overline{S}}} (f|_{S,z}*f|_{S,z})(x) - [x = \vec{0}]} < \epsilon
				\] for every set $S\subseteq [n]$ of size at most $k$. 
			\end{property}
		
			Convolutions are closely related to influences, so we have an additional influences property pertaining to an average restricted function:
		
			\begin{property}\label{property restriction influences}
				A boolean function $f:\F_2^n\to \spm$ with $\abs{\fourier{f}{0}} <\half$ has the \textbf{Restriction Influences Property} $RI(k,\epsilon)$ if the $\gamma$-Influences of the average restriction to subcubes of dimension at most $k$ are close to $\half$. 
				\[
					\abs{\exv_{z\in \F_2^{\overline{S}}} \Inf_{\gamma}[f|_{S,z}] - \half} < \epsilon
				\] for every set $S\subseteq [n]$ of size at most $k$ and every nonzero $\gamma\in \F_2^S$. 
			\end{property}

			The next combinatorial property states that many pairs of vertices on one side of the bipartite Cayley graph $BC(f)$ of $f$ (see Definition \ref{defn bipartite Cayley graph}) have the same number of shared neighbors.  
			\begin{property}\label{property local strong regularity}
				 A boolean function $f:\F_2^n\to \spm$ with $\abs{\fourier{f}{0}} <\half$ has the \textbf{Local Strong Regularity Property} $LSR(k,\epsilon)$ if the any pair of vertices within distance $k$ of each other in the left part of the bipartite Cayley graph of $f$ have approximately the same number of common neighbors, i.e. 
				\[
					\abs{\frac{N_{BC(f)}(u)\cap N_{BC(f)}(v)}{2^n} - \left(\frac{1}{4} - \frac{\fourier{f}{0}}{2}\right)} < \epsilon 
				\] for every pair of vertices $x,y$ on the left side of $BC(f)$ such that $\abs{x - y} \leq k$. 
			\end{property}
			Note that every bipartite Cayley graph has a symmetry which reverses its left and right parts, so we only need to consider the left part in the quasirandom property.  We remark that the assumption that $\abs{\fourier{f}{0}} <\half$ is required for Property \ref{property local strong regularity} and the other combinatorial properties.

		 	Our next property concerns the bipartite graph homomorphisms of a fixed bipartite graph with a fixed left map.  
			\begin{property}\label{property local embeddings}
				Define $p = \frac{1}{4} - \frac{\fourier{f}{0}}{2}$ and $q = \half - \frac{\fourier{f}{0}}{2}$. A boolean function $f:\F_2^n\to \spm$ with $\abs{\fourier{f}{0}} <\half$ has the \textbf{Degree $2$ Homomorphisms Property} $DTH(k,\epsilon)$ if every bipartite graph $G$ appears as a subgraph of $BC(f)$ as often as in a random bipartite graph for any choice of left map for $G$ with diameter at most $k$, as long as $G$ has maximum degree $2$ in its right part. More formally, the Degree-$2$ Homomorphisms Property holds if
				\[
					\abs{\bhom_\psi(G,BC(f)) - p^{r_2}q^{r_1}} < \epsilon p^{r_2}q^{r_1}
				\]  for every bipartite graph $G = (U\sqcup V,F)$ such that $V$ has maximum degree $2$ in $G$ and $|V| \leq 2^{n/2}$, and every injection $\psi:U\to \F_2^n$ of diameter at most $k$ where $r_2,r_1$ are the number of vertices in the right part of $G$ with degree $2$ and $1$ respectively. 
			\end{property}

			Our next combinatorial property considers embeddings into a the rainbow Hamming graph as defined in definitions \ref{defn colored multigraph} and \ref{defn rainbow embedding}. 
			
			\begin{property}\label{property rainbow mebeddings}
				A boolean function $f:\F_2^n\to\spm$ with $\abs{\fourier{f}{0}} <\half$ has the \textbf{Rainbow Embeddings Property} $RAIN(d,\epsilon)$ if for every fixed simple graph $G$ and every choice of injection $\phi$ from $G$ to the rainbow Hamming graph of $f$, there are close to $2^{-\abs{E(G)}}$ colorings of $G$ which become rainbow embeddings of $G$ under $\phi$. Namely, the Rainbow Embeddings Property holds if
				\[
				\abs{\emb_{\phi}(G,RHG(d,f)) - 2^{-\abs{E(G)}}} < \epsilon
				\] for every fixed graph $G$, and every $\phi:V(G)\hookrightarrow V(RHG(d,f))$ of diameter at most $d$.
			\end{property}
			\begin{rmk}
				In Figure \ref{figure rainbow Hamming graph example}, consider the graph $K_{1,2}$ with the injection which sends the left vertex to $00$ and the two right vertices to $01,10$. We note that there are exactly $4$ possible edge-colorings of $K_{1,2}$ which become rainbow embeddings under this injection. As there are $12$ possible colorings, it follows that $OR$ does not satisfy the Rainbow Embeddings Property $RAIN(d,\epsilon)$ for any $\epsilon < 1/6$. 
			\end{rmk}

			For properties $P(d,\epsilon)$ and $Q(d,\epsilon)$, we say $P$ \textit{implies} $Q$ with error bound $\delta$, denoted $P(d,\epsilon)\overset{\delta}{\implies} Q(d,\epsilon)$, if for every $d \geq 1$, every $\epsilon > 0$, and every boolean function $f:\F_2^n\to \spm$, there is a $\delta > 0$ such that 
			
			\[
				P(d,\delta) \implies Q(d,\epsilon)\hspace{1cm}
			\] where $\delta$ depends on $d$ and $\epsilon$ but not on the size of the domain of $f$. If 
			\[
				P(d,\epsilon) \overset{\delta_1}{\implies} Q(d,\epsilon)~~\text{and}~~Q(d,\epsilon) \overset{\delta_2}{\implies} P(d,\epsilon)
			\] for some error bounds $\delta_1$ and $\delta_2$, we say that $P$ and $Q$ are \textit{equivalent}. Our main result is as follows: 
			\begin{thm}\label{thm main result}
				For any fixed $d$ and $\epsilon$, the properties $INF(d,\epsilon)$(\ref{property directional influences}), $SD(d,\epsilon)$(\ref{property spectral discrepancy}), $RF(d,\epsilon)$(\ref{property restriction fourier}), $RC(d,\epsilon)$(\ref{property restriction convolutions}), $RI(d,\epsilon)$(\ref{property restriction influences}),  $LSR(d,\epsilon)$(\ref{property local strong regularity}), $DTH(d,\epsilon)$(\ref{property local embeddings}), and $RAIN(d,\epsilon)$(\ref{property rainbow mebeddings}) are equivalent.
			\end{thm}	
			
			If a boolean function $f$ satisfies any one of these properties for some $d$ and $\epsilon$, we say that $f$ is \textit{quasi-random} of rank $d$ with error bound $\epsilon$.

			We can summarize the proof of our main result in figure \ref{figure diagram of main proof}, where each arrow is labeled with the relevant theorem and error bound.
			\begin{figure}[h]
				\centering
				\begin{tikzpicture}
					\node (inf) at (0,2) {$INF(d,\epsilon)$};
					\node (sd) at (4,2) {$SD(d,\epsilon)$};
					\node (rf) at (4,0) {$RF(d,\epsilon)$};
					\node (rc) at (4,-2) {$RC(d,\epsilon)$};							
					\node (ri) at (2,-1) {$RI(d,\epsilon)$};					
					\node (lsr) at (-4,-2) {$LSR(d,\epsilon)$};
					\node (dte) at (-4,0) {$DTH(d,\epsilon)$};
					\node (rain) at (-4,2) {$RAIN(d,\epsilon)$};
					
					\draw[very thick,->] (inf) --  (sd) node[midway,above] {$\epsilon/2$} node[midway,below] {Thm (\ref{lem influences to spectral sample})};
					\draw[very thick,->] (sd) -- (rf) node[midway,left] {$\epsilon$} node[midway,right]  {Thm \ref{lem spectral sample to restrictions fourier}};
					\draw[very thick,->] (rf) -- (rc) node[midway,left] {$\epsilon/2^d$} node[midway,right] {Thm \ref{lem restriction Fourier to restriction convolution}};
					\draw[very thick,->] (rc) -- (ri) node[midway,right] {$\epsilon$} node[midway,left] { Thm \ref{thm restriction convolution to restriction influences}};
					\draw[very thick,->] (ri) -- (inf) node[midway,left] {Thm (\ref{thm restriction influences to influences}} node[midway,right] {$\epsilon$};
					\draw[very thick,->] (rc) -- (lsr) node[midway,above] {$4\epsilon$} node[midway,below] {Thm \ref{thm restriction convolution to local strong regularity}};
					\draw[very thick,->] (inf) -- (lsr) node[midway,left] {$4\epsilon$} node[midway,right] {Thm \ref{thm influence to local strong regularity}};
					\draw[very thick,->] (lsr) --  (dte) node[midway,left] {$\epsilon p/2e\abs{D_2}$} node[midway,right] { Thm \ref{thm local strong regularity to degree two embeddings}};
					\draw[very thick,->] (dte) -- node[midway,left] {$\epsilon\left(\frac{5}{2} + \fourier{f}{0}\right)^{-\abs{E(G)}}$} node[midway,right] {Thm (\ref{thm degree two embeddings to rainbow embeddings}} (rain) ;
					\draw[very thick,->] (rain) -- node[midway,below] {Thm \ref{thm rainbow embeddings to influences}} node[midway,above] {$\epsilon/4$} (inf) ;
				\end{tikzpicture}
				\caption{The implications in the main theorem (\ref{thm main result}). Each edge gives the loss in $\epsilon$ and the reference to the theorem in which the implication is shown.}
				\label{figure diagram of main proof}
			\end{figure}
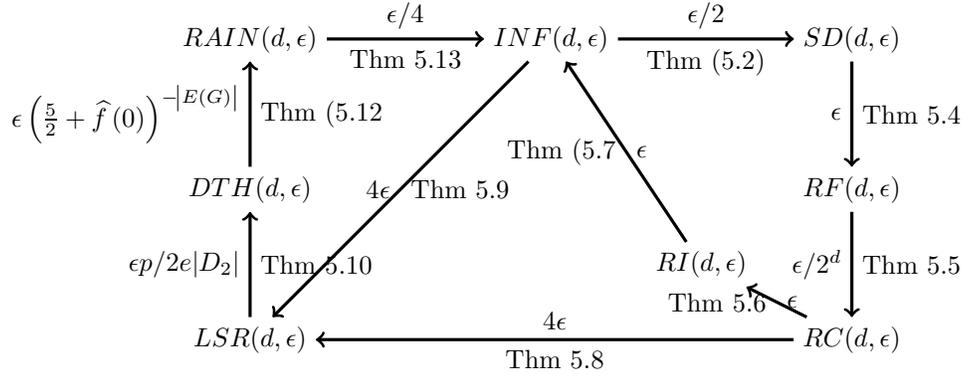
		
		One can easily observe that $P(d + 1,\epsilon) \implies P(d,\epsilon)$ for each property $P$ and every $d$ and $\epsilon$. Our second main result shows that these inclusions are strict, i.e. that there are functions which are quasi-random of rank $d$ but not quasi-random of rank $d + 1$.
		\begin{thm}\label{thm tower theorem}
			For each $d \geq 1$, there exists an explicit function $f_d:\F_2^n\to\spm $ with $\abs{\fourier{f}{0}} <\half$ such that 
			\begin{itemize}
				\item $f_d$ satisfies the Balanced Influences Property of rank $d$ with error $0$.
				\item $f_d$ does not satisfy the Balanced Influences Property of rank $d + 1$ for any bound  $\epsilon < \half$. 
			\end{itemize}
		\end{thm}
	
\section{Relating Quasirandom Boolean Functions to Extant theories}\label{sec previous theories}
		There are various other quasi-randomness theorems for boolean functions implicitly or explicitly considered in several related ranging from hypergraphs to analysis of boolean functions. Here we compare the properties defined in Section \secref{sec main result} to an incomplete list of these extant theories.

		We first consider Chung and Tetali's work on the relationship between $k$-uniform hypergraphs and boolean functions in \cite{chung1993communication}. Their ideas also appear implicitly in the works of Gowers \cite{Gowers2006} on hypergraph regularity lemmas and Szemeredi's Theorem, and in a survey paper by Casto-Silva \cite{castro-silva2021quasirandomness}. These works convert a boolean function to a $k$-uniform hypergraph via the following construction. Given a boolean function $f:\F_2^n\to \spm$, its \textit{Cayley Hypergraph} $H$ has the vertex set $\F_2^n$ and choosing hyper-edges $\{x_1,\dots,x_k\} \in E(H) \iff f(x_1 + \dots + x_k) = -1$. Via the Cayley hypergraph, these authors transfer the theory of quasi-randomness for uniform hypergraphs to boolean functions. 
		
		The central definition in these works is the following, which we state is a slightly non-traditional fashion: 
		\begin{defn}\label{defn Gowers norms}
			The \textit{$k$-th Gowers Uniformity Norm}, denoted $\gownorm{f}{k}$, is defined as
			\[
				\gownorm{f}{k} := \left(\exv_{x\in \F_2^n} \exv_{v_1,\dots,v_k\in \F_2^n} \prod_{\alpha_1,\dots,\alpha_k\in \{0,1\}} f(x + \alpha_1v_1 + \dots + \alpha_kv_k)\right)^{2^{-k}}
			\]
		\end{defn}
		We will typically use the following equivalent formula (see \cite{Hatami2019}):
		\[
			\gownorm{f}{k}  = \left(\exv_{x\in\F_2^n,M\in \F_2^{n\times k}}\prod_{v\in \F_2^k} f(x + Mv)\right)^{2^{-k}}.
		\]
		The Gowers norms are a direct translation of the properties in \cite{Chung1991a,chung1993communication} which count even and odd octahedra in $k$-Uniform hypergraphs.  
		For these theories, the key pseudorandom property is the following:
		\begin{property}\label{property smll gownorm}
			A boolean function $f:\F_2^n\to \spm$ is \textbf{$(\epsilon,d)$ $\F_2$-Regular} if
			\[
				\gownorm{f}{d + 1} < \epsilon
			\]
		\end{property}
		As shown in Castro-Silva's monograph \cite{castro-silva2021quasirandomness},  $(\epsilon,k + 1)$ $\F_2$-Regularity implies $(\epsilon,k)$ $\F_2$-Regularity, and the implication is strict. Hence, just as we have a hierarchy of quasirandom properties in our Theorem \ref{thm tower theorem}, we can view $(\epsilon,k)$ $\F_2$-Regularity as a similar hierarchy indexed by $k$. Furthermore, the $k + 1$-st Gowers norm controls correlation of $f$ with functions of $\F_2$-degree at most $k$ as shown in \cite{Hatami2019}. Hence, we can say that the hierarchy is indexed by $\F_2$-degree.
		
		We show the following Theorem whose proof is found in section \secref{sec comparisons}:
		\begin{thm}\label{thm relation to F2 regularity}
				Let $f:\F_2^n\to \spm$ have $(\epsilon,d)$-Balanced Influences. Then $f$ is $(\sqrt{2^{-d} + \epsilon},1)$ $\F_2$-Regular. 
				
				There is a function $f:\F_2^n\to \spm$ with $(0,n)$-Balanced Influences, yet $f$ is not $(\epsilon,2)$ $\F_2$-Regular for any $\epsilon < 1$.
				
				For any $d \geq 2$, there is a function $g:\F_2^n\to \spm$ which is $(\epsilon,d)$ $\F_2$-Regular and does not have the Balanced Influences Property of any rank $k \geq 1$ for any error bound $\epsilon < \half$.
		\end{thm}

		O'Donnell presents several pseudorandom properties in \cite{o'donnell_2014} which center on the Fourier expansion defined in Section \secref{sec defns}. The first pseudorandom property he mentions is the following:
		\begin{property}\label{property low degree fourier}
			A boolean function $f:\F_2^n\to \spm$ is \textbf{$(\epsilon,d)$ $\R$-Regular} if 
			\[
				\abs{\fourier{f}{\gamma}} < \epsilon
			\]for every $\gamma\in \F_2^n$ with $\abs{\gamma} \leq d$. 
		\end{property} 
		
		 One can easily verify that $(\epsilon,d + 1)$ $\R$-regularity implies $(\epsilon,d)$ $\R$-regularity, and a Fourier character of $\chi_{\gamma}$ where $\abs{\gamma} = d + 1$ shows that the implication is strict. Hence, just as we have a hierarchy of quasirandom properties in our Theorem \ref{thm tower theorem},  $(\epsilon,k)$ $\R$-Regularity can be viewed as having an increasing hierarchy of pseudorandom properties indexed by $k$.  Furthermore, $(\epsilon,n)$ $\R$-regularity is equivalent to $(\epsilon,1)$ $\F_2$-regularity as is shown in Proposition 6.7 of \cite{o'donnell_2014}. 
		
		We show the following Theorem whose proof can be found in section \secref{sec comparisons}:
		\begin{thm}\label{thm relation to R-regularity}
				Let $f:\F_2^n\to \spm$ have $(\epsilon,k)$ Balanced Influences. Then $f$ is $(\sqrt{2^{-k} + \epsilon},n)$ $\R$-regular. 
				
				Conversely, there is a function which is $(\epsilon,n)$ $\R$-regular which does not have $(\epsilon,k)$ Balanced Influences for any rank $k \geq 1$ or error bound $\epsilon < \half$. 
		\end{thm}

		Another collection of pseudorandom properties of Boolean functions appears implicitly in Chung and Graham's work on quasirandom subsets of $\Zp{N}$ \cite{Chung1992}. To apply their work to boolean functions, we can identify the set of binary strings with elements of $\Zp{2^n}$. Then a boolean function can be identified with the set of elements of $\Zp{2^n}$ on which it takes the value $-1$.   Their key pseudorandom property is the following:
		\begin{property}\label{property mod 2^n regularity}
			A boolean function $f:\Zp{2^n}\to \spm$ is \textbf{$\Zp{2^n}$-Regular} with error bound $\epsilon$ if $f$ has correlation at most $\epsilon$ with the all nonzero characters of $\Zp{2^n}$, i.e. for every nonzero $j\in \Zp{2^n}$,
			\[
				\abs{\exv_{z\in \Zp{2^n}} f(z)\exp(2\pi i j z/2^n)} < \epsilon.
			\]
		\end{property}
		
		As shown by Chung and Graham \cite{Chung1992}, $\Zp{2^n}$-Regularity controls the correlations of a function $f$ with a shifted copy of itself much like our Balanced Influences Property (see \ref{property directional influences}). However, the arithmetic operations considered in $\Zp{2^n}$-Regularity are carried out over $\Zp{2^n}$ rather than $\F_2^n$ as in the Balanced Influences Property.

		We prove the following Theorem whose proof is found in section \secref{sec comparisons}:
		\begin{thm}\label{thm relation to Z/2^nZ regularity}
			For any $\epsilon > 0$ there is a function which is $\Zp{2^n}$-Regular with error bound $\epsilon$ but is not $(\delta,n - k + 1)$ $\R$-Regular for any $\delta < 1$ where $k = -C_0\ln(1/\epsilon)$ for some absolute constant $C_0$. 
		\end{thm}

		O'Donnell \cite{o'donnell_2014} defines an additional pseudo-random property which uses a different generalization of influences as follows.
		\begin{defn}\label{defn small stable influences}
			For a coordinate $i$ and a parameter $\rho\in [0,1]$, the \textbf{$\rho$-stable influence} of a boolean function $f:\F_2^n\to \spm$ is
			\[
				\Inf_{i}^{\rho}[f] := \sum_{\substack{\gamma\in \F_2^n\\ \gamma_i = 1}} \rho^{\abs{\gamma} - 1}\fourier{f}{\gamma}^2.
			\]
		\end{defn}

		The key pseudorandom property is:
		\begin{property}\label{property small stable influences}
			A boolean function $f:\F_2^n\to \spm$ has \textbf{$(\epsilon,\rho)$ Small Stable Influences} if
			\[
				\Inf_{i}^{\rho}[f] < \epsilon
			\] for every $i\in [n]$.
		\end{property}
	
		As shown by O'Donnell\cite{o'donnell_2014}, $\rho$-Stable-Influences measure the expected change in the function if the input bits are changed via a particular noise model. Thus, $(\epsilon,\rho)$ Small Stable Influences implies a form of noise stability.

		We show the following Theorem whose proof is found in section \secref{sec comparisons}:
		\begin{thm}\label{thm relation to small stable influences}
			Let $f:\F_2^n\to \spm$ satisfy $SD(d,\epsilon e^{(d - 1)((\delta/2) - \ln(2))})$. Then, $f$ has $(\epsilon,\delta)$ small stable influences. 
			Furthermore, $\chi_{\bone}$ has $((1 - \delta)^n,\delta)$ small stable influences for any $\delta < 1$ but does not have $(\epsilon,k)$ Balanced Influences for any $k$ and any $\epsilon < \half$. 
		\end{thm}

		To summarize, we have figure \ref{figure relations between different theories} which includes the relationships between each theory of quasirandomness and our results in Section \secref{sec main result}.
		\begin{figure}[h]
			\centering
			\begin{tikzpicture}
			\tikzstyle{statement} = [rectangle,draw=black,fill = white,minimum width= 2cm,minimum height= 1cm];
			\tikzstyle{Implication} = [line width=1mm,draw=black!50];
			\tikzstyle{Incomparable} = [dotted,line width=1mm,draw=red!50];
			
			\node[statement] (gow) at (3,4) {$(\epsilon,d)$ $\F_2$-Regular, $d > 1$};
			\node[statement] (reg) at (0,0) {$(\epsilon,1)$ $\F_2$-Regular};
			\node[statement] (ssinf) at (-4,-2) {$(\epsilon,\delta)$-Small Stable Influences};
			\node[statement] (zzreg) at (4,-1) {$\Zp{2^n}$-Regularity};
			\node[statement] (ldr) at (0,-4) {$(\epsilon,k)$ $\R$-Regular, $k < n$};
			\node[statement,draw=blue] (inf) at (-4,2) {\textcolor{blue}{\textbf{$(\epsilon,d)$ Balanced Influences, $d > 0$}}};
			\node[statement] (bent) at (-4,4) {Bent};

			\draw[->,Implication] (gow) -- (reg);
			\draw[->,Implication] (reg) -- (ldr);
			\draw[->,Implication,draw=blue] (inf) -- node[midway,right] {Theorem \ref{thm relation to R-regularity}} (reg);
			\draw[->,Implication] (ssinf) -- node[midway,left] {Ex 6.5f \cite{o'donnell_2014}} (ldr);
			\draw[->,Implication,draw=blue] (inf) -- node[midway,left] {Theorem \ref{thm relation to small stable influences}}(ssinf);
			\draw[->,Implication] (bent) -- (inf);

			\draw[Incomparable,<->] (reg) -- node[midway,above left] {Ex 6.5d \cite{o'donnell_2014}} node[solid,midway,cross,sloped,minimum width =12pt,minimum height =12pt] {} (ssinf);
			\draw[Incomparable,->] (gow) -- node[midway,above left,color=blue] {Theorem \ref{thm relation to F2 regularity}} node[solid,midway,cross,sloped,minimum width =12pt,minimum height =12pt] {} (inf);
			\draw[Incomparable,->] (bent) -- node[midway,above,color=blue] {Theorem \ref{thm relation to F2 regularity}}  node[solid,midway,cross,sloped,minimum width =12pt,minimum height =12pt] {} (gow);
			\draw[Incomparable,<-] (ldr) -- node[midway,below right,color=blue] {Theorem \ref{thm relation to Z/2^nZ regularity}} node[solid,midway,cross,sloped,minimum width =12pt,minimum height =12pt] {} (zzreg);
			\end{tikzpicture}
			\caption{The relationships between different theories of quasi-randomness. Each box is a distinct theory of quasi-randomness. Each arrow is a strict implication. Beside each arrow we give a reference to the proof of the implication. If the result follows by definition, we omit the label for the sake of space. The results of this paper are in bold blue text and blue arrows. Non-implications are red dotted lines with an $X$ in the middle, with a citation for each result.}
			\label{figure relations between different theories}
		\end{figure}
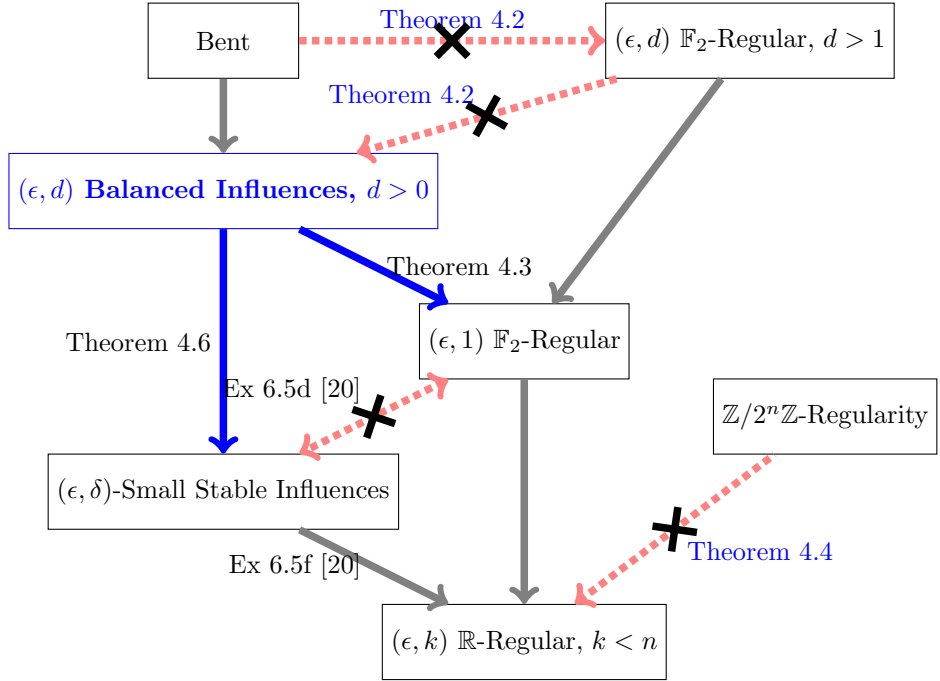
	
	\section{Proof of the Equivalences in Theorem \ref{thm main result}}\label{sec main proof}			
			We now proceed to the proof of our main theorems.  The following property of the directional influence will be quite useful later. 
		\begin{prop}\label{lem influence is autocorrellation is average value of derivative}
			If $\gamma\in \F_2^n$, then
			\[
			 f*f(\gamma) = 1 - 2\Inf_{\gamma}[f] 
			\]
		\end{prop}
		\begin{proof}
			By definition of $\gamma$-influence,
			\begin{align*}
			1 - 2\Inf_{\gamma}[f] &= 1 - 2\prob[f(x) \neq f(x + \gamma)]\\
			&= \exv_{x\in\F_2^n} \left(1 - 2[f(x) \neq f(x + \gamma)]\right)\\
			&= \exv_{x\in\F_2^n} f(x)f(x + \gamma)\\
			&= f*f(\gamma) 
			\end{align*}
			where we use the fact that $f(x)\in \spm$ in the second and third equalities.
		\end{proof}

		Our main Theorem \ref{thm main result} will be proved in a sequence of implications each of which we state as a Theorem. The first Theorem in the proof of Theorem \ref{thm main result} relates the Balanced Influences Property to the Spectral Discrepancy Property. 	
		\begin{thm}\label{lem influences to spectral sample}
			For any fixed integer $d \geq 1$ and any $\epsilon > 0$, the Balanced Influences Property $INF(d,\epsilon/2)$ implies the Spectral Discrepancy Property $SD(d,\epsilon)$.
		\end{thm}	
			\begin{proof}
			Fix a subcube $C(S,z_0)$ where $\abs{S} = n - k$ where $k \leq d$. Let $M\in \F_2^{\overline{S}\times [n]}$ be the projection matrix which sends $z\in \F_2^n$ to $z|_{\overline{S}}$. 
			
			We observe that the indicator function $[\gamma\in C(S,z_0)]$ can be written as
			\begin{equation}
				[\gamma\in C(S,z_0)] = \exv_{v\in \F_2^{\overline{S}}} (-1)^{v\cdot (M\gamma- z_0)}.\label{eqn indicator function as exp sum}
			\end{equation} Indeed, if $\gamma\in C(S,z_0)$, then $M\gamma = z_0$, and $\exv_{v\in \F_2^{\overline{S}}} (-1)^{v\cdot (M\gamma- z_0)} = \exv_{v\in \F_2^{\overline{S}}} 1 = 1$. If $\gamma\notin C(S,z_0)$, then $\gamma_j \neq (z_0)_j$ for some $j\in \overline{S}$. Therefore, $\exv_{v\in \F_2^{\overline{S}}} (-1)^{v\cdot (M\gamma- z_0)} = \exv_{v\in \F_2^{\overline{S}}} (-1)^{v\cdot y}$ for some nonzero vector $y$. Hence, $\exv_{v\in \F_2^{\overline{S}}} (-1)^{v\cdot (M\gamma- z_0)} = 0$.  Let $f$ be a function which satisfies the Balanced Influence Property $INF(d,\epsilon/2)$. To use Equation (\ref{eqn indicator function as exp sum}), we expand the definition of the spectral sample.
			\begin{align*}
			\prob_{\gamma\sim \mathcal{S}_f}\left[\gamma\in C(S,z_0)\right] &= \sum_{\gamma\in C(S,z_0)} \fourier{f}{\gamma}^2 \nonumber\\
			&= \sum_{\gamma\in \F_2^n} \fourier{f}{\gamma}^2[\gamma\in C(S,z_0)] \nonumber\\
			&= \sum_{\gamma\in \F_2^n} \fourier{f}{\gamma}^2\exv_{v\in \F_2^{\overline{S}}} (-1)^{v\cdot (M\gamma- z_0)} ~~~~ \text{By Equation (\ref{eqn indicator function as exp sum})}\\
			&= 	\exv_{v\in \F_2^{\overline{S}}} (-1)^{v\cdot z_0}\sum_{\gamma\in \F_2^n} \fourier{f}{\gamma}^2 (-1)^{v\cdot M\gamma} \nonumber\\
			&= \exv_{v\in \F_2^{\overline{S}}} (-1)^{v\cdot z_0}\sum_{\gamma\in \F_2^n} \fourier{f}{\gamma}^2 (-1)^{M^\top v\cdot \gamma} \nonumber\\
			&= \exv_{v\in \F_2^{\overline{S}}} (-1)^{v\cdot z_0}\sum_{\gamma\in \F_2^n} \fourier{f}{\gamma}^2 \chi_{\gamma}(M^\top v) ~~~~~ \text{By definition of $\chi_\gamma$}\\
			&= \exv_{v\in \F_2^k}  (-1)^{v\cdot z_0}f*f(M^\top v) \label{eqn use defn of convolution}
			\end{align*} where we use the Fourier expansion of $f*f$ in the final line. Notice that $f*f(M^\top 0) = (f*f)(0) = 1$, and that $x = 0$ is the only solution to $M^\top x = 0$. Therefore, we can write
			\begin{align*}
					\abs{\prob_{\gamma\sim \mathcal{S}_f}\left[\gamma\in C(S,z_0)\right]- 2^{-k}} &= \abs{\sum_{v\in \F_2^k}  (-1)^{v\cdot z_0}\dfrac{f*f(M^\top v)}{2^k}  - 2^{-k}}\\
					&= \abs{\sum_{v\in \F_2^k\setminus\{\vec{0}\}}  (-1)^{v\cdot z_0}\dfrac{f*f(M^\top v)}{2^k}}\\
					&\leq \frac{1}{2^k}\sum_{v\in \F_2^k\setminus\{\vec{0}\}} \abs{f*f(Mv)}\\
					&= \frac{1}{2^k}\sum_{v\in \F_2^k\setminus\{\vec{0}\}} \abs{1 - 2\Inf_{M^\top v}[f]}
			\end{align*}
			where we use Lemma \ref{lem influence is autocorrellation is average value of derivative} in the final line. As $k \leq d$, we have $\abs{S} = n - \abs{\overline{S}} = k \leq d$, $\abs{S} \leq d$. Thus, $\abs{v} \leq d$. Since $M$ is a projection matrix, $\abs{M^\top v} = \abs{v} \leq d$ . There fore, we may apply $INF(d,\epsilon/2)$ to find
			\begin{align*}
			\abs{\prob_{\gamma\sim \mathcal{S}_f}\left[\gamma\in C(S,z_0)\right]- 2^{-k}} 
			&\leq \frac{1}{2^k}\sum_{v\in \F_2^k\setminus\{\vec{0}\}} \epsilon \leq \epsilon
			\end{align*} As $C(S,z_0)$ is arbitrary, $f$ also satisfies the Spectral Discrepancy Property $SD(d,\epsilon)$.
		\end{proof}

		For our next few implications, we will need the following result from O'Donnell's book \cite{o'donnell_2014}, translated into our notation.
		\begin{lem}\label{lem Fouriercoeffs of restrictions}[\cite{o'donnell_2014} Proposition 3.21]
			If $C(S,z)$ is a fixed subcube and $\gamma\in \F_2^S$, then
			\[
			\fourier{f|_{S,z}}{\gamma} = \sum_{\delta \in \F_2^{\overline{S}}} \fourier{f}{\delta \underset{S}{\otimes} \gamma}\chi_{\delta}(z)
			\]
		\end{lem}

		Now we can relate the spectral sample to the Fourier coefficients of restricted functions. 
		\begin{thm}\label{lem spectral sample to restrictions fourier}
			For any fixed $d \geq 1$ and $\epsilon > 0$ the Spectral Discrepancy Property $SD(d,\epsilon)$ implies the Restriction Fourier Property $RF(d,\epsilon)$.
		\end{thm}
		\begin{proof}
			This proof is essentially the proof of Corollary 3.22 in \cite{o'donnell_2014}, which we reproduce here in our notation for completeness. Suppose $f:\F_2^n\to \spm$ satisfies the Spectral Discrepancy Property $SD(d,\epsilon)$. Let $C(S,z)$ be an arbitrary subcube of dimension $k$ where $k \leq d$. Then for a fixed $\gamma\in \F_2^S$, Lemma \ref{lem Fouriercoeffs of restrictions} gives us
			\begin{align}
				\exv_{z\in \F_2^{\overline{S}}} \fourier{f|_{S,z}}{\gamma}^2 &= \exv_{z\in \F_2^{\overline{S}}} \left(\sum_{\delta \in \F_2^{\overline{S}}} \fourier{f}{\delta \underset{S}{\otimes} \gamma}\chi_{\delta}(z)\right)^2\nonumber\\
				&= \sum_{\delta_1,\delta_2\in \F_2^{\overline{S}}}  \exv_{z\in \F_2^{\overline{S}}} \fourier{f}{\delta_1 \underset{S}{\otimes} \gamma}\fourier{f}{\delta_2 \underset{S}{\otimes} \gamma}\chi_{\delta_1}(z)\chi_{\delta_2}(z)\nonumber\\
				&= 	\sum_{\delta_1,\delta_2\in \F_2^{\overline{S}}} \fourier{f}{\delta_1 \underset{S}{\otimes} \gamma}\fourier{f}{\delta_2 \underset{S}{\otimes} \gamma}\exv_{z\in \F_2^{\overline{S}}} \chi_{\delta_1}(z)\chi_{\delta_2}(z)\nonumber\\
				&= \sum_{\delta\in \F_2^{\overline{S}}} \fourier{f}{\delta \underset{S}{\otimes} \gamma}^2\label{eqn sd to rf orthognality}\\
				&= \prob_{\eta\sim \mathcal{S}_f} \left[\eta\in C(\overline{S},\gamma)\right]\label{eqn sd to rf spectral sample def}
			\end{align} where we use the orthogonality of the Fourier characters in line (\ref{eqn sd to rf orthognality}) and the definition of the spectral sample line (\ref{eqn sd to rf spectral sample def}). As $k \leq d$, $\abs{\overline{S}} = n - k \geq n - d$.  Thus we can apply Property $SD(d,\epsilon)$ to $C(\overline{S},z)$ to find that
			\[
				\abs{\prob_{\eta\sim \mathcal{S}_f} \left[\eta\in C(\overline{S},\gamma)\right] - 2^{-k}} < \epsilon
			\] for every $\gamma\in \F_2^S$. Hence,
			\[
				\abs{\exv_{z\in \F_2^{\overline{S}}} \fourier{f|_{S,z}}{\gamma}^2 - 2^{-k}} < \epsilon
			\] for every $\gamma\in \F_2^S$. As $C(S,z)$ is arbitrary, $f$ also satisfies the Restriction Fourier Property $RF(d,\epsilon)$.
		\end{proof}
		
		With a bound on the Fourier coefficients of restricted functions, we can bound the convolution of a restricted function with itself.
		\begin{thm}\label{lem restriction Fourier to restriction convolution}
				For any fixed $d \geq 1$ and $\epsilon > 0$ the Restriction Fourier Property $SD(d,\epsilon/2^d)$ implies the Restriction Convolution Property $RC(d,\epsilon)$
		\end{thm}
		\begin{proof}
			Let $f:\F_2^n\to \spm$ have the Restriction Fourier Property $RF(d,\epsilon/2^d)$, and note that $f$ also satisfies $RF(k,\epsilon/2^k)$ for every $k \leq d$. Fix $k\in \N$ such that $k \leq d$ and a set $S\subseteq [n]$ where $|S|= k$.
			
			 We have
			\begin{align*}
				\exv_{z\in \F_2^{\overline{S}}} (f|_{S,z}*f|_{S,z})(x) &= \exv_{z\in \F_2^{\overline{S}}} \sum_{\delta\in \F_2^S} \fourier{f|_{S,z}}{\delta}^2\chi_{\delta}(x)\\
				&=  \sum_{\delta\in \F_2^S} \left(\exv_{z\in \F_2^{\overline{S}}} \fourier{f|_{S,z}}{\delta}^2\right)\chi_{\delta}(x)
			\end{align*}
			Using the Fourier expansion of the indicator function $[x = 0]$, we then have
			\begin{align*}
				\abs{\exv_{z\in \F_2^{\overline{S}}} (f|_{S,z}*f|_{Sz})(x) - [x = 0]} &= \abs{\sum_{\delta\in \F_2^S} \left(\exv_{z\in \F_2^{\overline{S}}} \fourier{f|_{S,z}}{\delta}^2 - \frac{1}{2^k}\right)\chi_{\delta}(x)}\\
				&\leq \sum_{\delta\in \F_2^S} \abs{\exv_{z\in \F_2^{\overline{S}}} \fourier{f|_{S,z}}{\delta}^2 - \frac{1}{2^k}}\\
				&\leq \sum_{\delta\in \F_2^S} \frac{\epsilon}{2^k}\\
				&\leq \epsilon
			\end{align*} where we use $RF(k,\epsilon/2^k)$ in the penultimate line. Since $k$ and $S$ are arbitrary, we conclude that $f$ also satisfies the Restriction Convolution Property $RC(d,\epsilon)$.
		\end{proof}
	
		Lemma (\ref{lem influence is autocorrellation is average value of derivative}) gives us a further property:
		\begin{thm}\label{thm restriction convolution to restriction influences}
				For any fixed $d \geq 1$ and $\epsilon > 0$ the Restriction Convolution Property $SD(d,2\epsilon)$ implies the Restriction Influences Property $RF(d,\epsilon)$.
		\end{thm}
		\begin{proof}
			Suppose $f$  satisfies the Restriction Convolution Property $RC(d,2\epsilon)$. By Lemma (\ref{lem influence is autocorrellation is average value of derivative}) applied to $f|_{S,z}$, we have
			\[
				 \Inf_{\gamma}[f|_{S,z}] = \frac{1 - f|_{S,z}*f|_{S,z}(\gamma)}{2}
			\] for any fixed $S$ and $z$. Now fix $k\in \N$ such that $k \leq d$ and $S\subseteq [n]$ where $|S| = k$. Then,
			\begin{align*}
					\abs{\exv_{z\in \F_2^{\overline{S}}} \Inf_{\gamma}[f|_{S,z}] - \half} &= \abs{\left(\exv_{z\in \F_2^{\overline{S}}} \frac{1 - f|_{S,z}*f|_{S,z}(\gamma)}{2}\right) - \half}\\
					&=\abs{\exv_{z\in \F_2^{\overline{S}}} \frac{f|_{S,z}*f|_{S,z}(\gamma)}{2}}
			\end{align*} As $\gamma \neq 0$, $RC(d,2\epsilon)$ implies that the above is at most $\epsilon$. Hence, $f$ satisfies the Restriction Influences Property $RI(d,\epsilon)$.
		\end{proof}
	
		As one might imagine, the influences of restricted functions relate nicely to the influences of the original function. 
		\begin{thm}\label{thm restriction influences to influences}
		For any fixed $d \geq 1$ and $\epsilon > 0$, the Restriction Influences Property $RI(d,\epsilon)$ implies the Balanced Influences Property $INF(d,\epsilon)$.
		\end{thm}
		\begin{proof}
			Suppose $f:\F_2^n\to \spm$ satisfies the Restriction Influences Property $RF(d,\epsilon)$. Fix $S\subseteq [n]$ with $|S| \leq d$ and $\gamma\in \F_2^S$. Then,
			\begin{align*}
				\exv_{z\in \F_2^{\overline{S}}} \Inf_{\gamma}[f] &= \exv_{z\in \F_2^{\overline{S}}} \exv_{x\in \F_2^S} [f|_{S,z}(x + \gamma) \neq f|_{S,z}(x)]\\
				&= \exv_{z\in \F_2^{\overline{S}}} \exv_{x\in \F_2^S} [f(x \underset{S}{\otimes} z + \gamma\underset{S}{\otimes} z))\neq f(x\underset{S}{\otimes} z)]
			\end{align*}
			Let $y = x \underset{S}{\otimes} z$ and $\delta = \gamma\underset{S}{\otimes} 0$. Note that $\abs{\delta} \leq d$ as $\abs{S} \leq d$.
			\begin{align*}	
				&= \exv_{y\in \F_2^n} [f(y + \delta) \neq f(y)]\\
				&= \Inf_{\delta}[f]
			\end{align*} Since any vector of Hamming weight at most $d$ can be represented as $\gamma\underset{S}{\otimes} 0$ for some set $S$ with $\abs{S} \leq d$ and $\gamma\in \F_2^S$, the claim follows by $RI(d,\epsilon)$. Thus $f$ also satisfies the Balanced Influences Property $INF(d,\epsilon)$.
		\end{proof}
	
	Now we can cross over to more combinatorial properties.
	\begin{thm}\label{thm restriction convolution to local strong regularity}
		For any fixed $d \geq 1$ and $\epsilon > 0$, the Restriction Convolution Property $RC(d,4\epsilon)$ implies that Local Strong Regularity Property $LSR(d,\epsilon)$.
	\end{thm}
	\begin{proof}
		Suppose $f:\F_2^n\to \spm$ satisfies the Restriction Convolution Property $RC(d,4\epsilon)$. Fix $u\text{ and }v$ in the left part of the bipartite Cayley graph of $f$ such that $0 < \abs{u - v} \leq d$. Let $S\subseteq [n]$ with $\abs{S} = d$ be a set which contains all the indices on which $u$ and $v$ differ. Then,
		\begin{align*}
			\abs{\frac{\abs{N(u) \cap N(v)}}{2^n} - \frac{1}{4} + \frac{\fourier{f}{0}}{2}} &= \abs{\exv_{c\in \F_2^n} \frac{(1 - f(u + c))(1 - f(v + c))}{4} - \frac{1}{4} + \frac{\fourier{f}{0}}{2}}\\
			&= \frac{1}{4}\abs{2\fourier{f}{0} - \left(\exv_{c\in \F_2^n} f(u + c) + f(c + v)\right) + \exv_{c\in \F_2^n} f(u + c)f(v + c)}\\
			&= \frac{1}{4} \abs{\exv_{c\in \F_2^n} f(u + c)f(v + c)}
		\end{align*}
		where we use the definition $\fourier{f}{0}$ in the final line. Now write $u = x\underset{S}{\otimes} z$, $v = y\underset{S}{\otimes} z$, and $c = w\underset{S}{\otimes} a$. We then have
		\begin{align*}	
			\abs{\frac{\abs{N(u) \cap N(v)}}{2^n} - \frac{1}{4} + \frac{\fourier{f}{0}}{2}}
			&=  \frac{1}{4} \abs{\exv_{a\in \F_2^{\overline{S}}} \exv_{w\in \F_2^{S}} f(x\underset{S}{\otimes} z + w\underset{S}{\otimes} a)f(y\underset{S}{\otimes} z + w\underset{S}{\otimes} a)}\\
			&= \frac{1}{4} \abs{\exv_{w\in \F_2^{S}} \exv_{a\in \F_2^{\overline{S}}}  f(9x + w)\underset{S}{\otimes}(z + a))f((y + w)\underset{S}{\otimes}(z + a))}\\
			\intertext{Noting that $w + w = 0$,}
			&=  \frac{1}{4} \abs{\exv_{w\in \F_2^{S}} \exv_{a\in \F_2^{\overline{S}}} f|_{S,a + z}*f|_{S,a + z}(x + y)}\\
			&= \frac{1}{4} \abs{\exv_{a'\in \F_2^{\overline{S}}} f|_{S,a'}*f|_{S,a'}(x +  y)}\\
			&\leq \epsilon
		\end{align*} where we use $RC(d,4\epsilon)$ and the fact that $u \neq v$ in the final line. Thus $f$ also satisfies the Local Strong Regularity Property $LSR(d,\epsilon)$.
	\end{proof}
	
	\begin{thm}\label{thm influence to local strong regularity}
		For any fixed $d\geq 1$ and $\epsilon > 0$, the Balanced Influences Property $INF(d,4\epsilon)$ implies the Local Strong Regularity Property $LSR(d,\epsilon)$.
	\end{thm}
	\begin{proof}
		Suppose $f;\F_2^n\to \spm$ satisfies the Balanced Influences Property $INF(d,4\epsilon)$. Fix $u,v$ in the left part of the bipartite Cayley graph of $f$ such that $0 < \abs{u - v} \leq d$. Then,
		\begin{align*}
			\abs{\frac{\abs{N(u) \cap N(v)}}{2^n} - \frac{1}{4} + \frac{\fourier{f}{0}}{2}} &= \abs{\exv_{z\in \F_2^n} \frac{(1 - f(u + z))(1 - f(v + z))}{4} - \frac{1}{4} + \frac{\fourier{f}{0}}{2}}\\
			&= \frac{1}{4}\left(2\fourier{f}{0} - \left(\exv_{z\in \F_2^n} f(u + z) + f(z + z)\right) + \exv_{z\in \F_2^n} f(u + z)f(v + z)\right)\\
			&= \frac{1}{4}f*f(u + v)\\
			&\leq \epsilon
		\end{align*} where we use Lemma (\ref{lem influence is autocorrellation is average value of derivative}) in the third line and $INF(d,4\epsilon)$ in the final line. It follows that $f$ also satisfies the Local Strong Regularity Property $LSR(d,\epsilon)$.
	\end{proof}

	For $a,b\in \N$, let $(a)_b$ denote the \textbf{Falling Factorial} defined by $(a)_b = a(a - 1)\dots(a - b + 1)$. Recall that $S\hookrightarrow T$ is the set of injective functions from $S$ to $T$. For a graph $G$ two vertices $x\in V(G)$ and $y\in V(G)$, let $x\underset{G}{\sim} y$ denote the proposition that $x$ is adjacent to $y$. If the graph is clear from context, we will simply write $x\sim y$.  
	
	\begin{thm}\label{thm local strong regularity to degree two embeddings}
		 For any fixed $d\geq 1$ and $\epsilon > 0$, the Local Strong Regularity Property $LSR(d,\frac{\epsilon p}{2e\abs{D_2}})$ implies the Degree-$2$ Homomorphisms Property $DTE(d,\epsilon)$ where $D_2$ is the set of vertices of degree $2$ in the right part of a bipartite graph $G$..
	\end{thm}
	\begin{proof}
		 Fix a bipartite graph $G$ with $l$ vertices in its left part $L(G)$, $r$ vertices in its right part $R(G)$, and maximum degree $2$ in $R(G)$. Let $D_1,D_2$ denote sets of vertices in $R(G)$ of degree $1$ and $2$ respectively. Finally, for a vertex $r$ in $R(G)$, define $p_r := \begin{cases}
		q & \deg(r) = 1\\
		p & \deg(r) = 2
		\end{cases}$ where $p = \frac{1}{4} - \frac{\fourier{f}{0}}{2}$ and $q = \frac{1 - \fourier{f}{0}}{2}$.
		
		Suppose that $f:\F_2^n\to \spm$ satisfies the Local Strong Regularity Property $LSR(d,\frac{\epsilon p}{2e\abs{D_2}})$.  We begin by expanding the Degree-$2$ Homomorphisms Property (Property \ref{property local embeddings})
		\begin{align*}
		\abs{\bhom_\psi(G,BC(f)) - p^{\abs{D_2}}q^{\abs{D_1}}} &= \abs{\exv_{\phi:R(G)\hookrightarrow \F_2^n} \prod_{(l,r)\in E(G)} [\psi(l) \sim \phi(r)] - p^{\abs{D_2}}q^{\abs{D_1}}}\\
		&= 	\abs{\exv_{\phi:R(G)\hookrightarrow \F_2^n} \prod_{r\in R(G)} \prod_{l\in N_G(r)}  [\psi(l) \sim \phi(r)] - p^{\abs{D_2}}q^{\abs{D_1}}}
		\end{align*} where we use the fact that $G$ is bipartite in the second line.  We can now add $p_r - p_r$ to each term in the product and then telescope as follows:
		\begin{align*}
			\abs{\bhom_\psi(G,BC(f)) - p^{\abs{D_2}}q^{\abs{D_1}}} &= \abs{\exv_{\phi:R(G)\hookrightarrow \F_2^n} \left(\prod_{r\in R(G)} \left(\left(\prod_{l\in N_G(r)}  [\psi(l) \sim \phi(r)]\right) - p_r\right) + p_r\right) - p^{\abs{D_2}}q^{\abs{D_1}}}\\
			&= \abs{\sum_{\emptyset \neq S\subseteq R(G)} \exv_{\phi:R(G)\hookrightarrow \F_2^n}  \left(\prod_{r\in R(G)\setminus S}  p_r\right)\left(\prod_{r\in S} \left(\left(\prod_{l\in N_G(r)}  [\psi(l) \sim \phi(r)]\right) - p_r\right)\right)}\\
			&\leq \sum_{\emptyset \neq S\subseteq R(G)} \left(\prod_{r\in R(G)\setminus S}  p_r\right)\abs{\exv_{\phi:R(G)\hookrightarrow \F_2^n} \prod_{r\in S} \left(\left(\prod_{l\in N_G(r)}  [\psi(l) \sim \phi(r)]\right) - p_r\right)}\\
		\end{align*} We now focus on an individual set $S\subseteq R(G)$ and its contribution to the sum:
		\[
			X_S := \abs{\exv_{\phi:R(G)\hookrightarrow \F_2^n} \prod_{r\in S} \left(\left(\prod_{l\in N_G(r)}  [\psi(l) \sim \phi(r)]\right) - p_r\right)}
		\] We would like to exchange the expectation and the product over elements of $S$. However, the expectation is over all injections not all functions. Thus we proceed as follows.
		\begin{align*}
			X_S &= \frac{1}{(2^n)_s}\abs{\sum_{\phi:R(G)\hookrightarrow \F_2^n} \prod_{r\in S} \left(\left(\prod_{l\in N_G(r)}  [\psi(l) \sim \phi(r)]\right) - p_r\right)}\\
			&\leq \frac{1}{(2^n)_s}\abs{\sum_{\phi:R(G)\to \F_2^n} \prod_{r\in S} \left(\left(\prod_{l\in N_G(r)}  [\psi(l) \sim \phi(r)]\right) - p_r\right)}\\
			&= \frac{2^{ns}}{(2^n)_s}\prod_{i = 1}^{s} \abs{\exv_{\phi(r_i)\in \F_2^n}  \left(\left(\prod_{l\in N_G(r_i)}  [\psi(l) \sim \phi(r_i)]\right) - p_r\right)}
		\end{align*}  If $\deg(r_i) = 1$, then
		\begin{align*}
				\exv_{\phi(r_i)\in \F_2^n}  \left(\left(\prod_{l\in N_G(r_i)}  [\psi(l) \sim \phi(r_i)]\right) - p_r\right) &= \exv_{\phi(r_i)\in \F_2^n}    [\psi(l) \sim \phi(r_i)] - q\\
				&= \exv_{\phi(r_i)\in \F_2^n} \dfrac{1 - f(\psi(l) + \phi(r))}{2} - \frac{1  -\fourier{f}{0}}{2}\\
				&= 0
		\end{align*} Thus, $X_S$ is $0$ if $S$ contains any vertices of degree $1$. Assume then that every $r_i\in S$ has exactly $2$ neighbors. Since the neighbors of $\phi(r_i)$ are fixed by $\psi$, we may apply  $LSR(d,\frac{\epsilon p }{2e\abs{D_2}})$ to conclude that
		\[
			\abs{\exv_{\phi(r_i)\in \F_2^n}  \left(\left(\prod_{l\in N_G(r_i)}  [\psi(l) \sim \phi(r_i)]\right) - p_{r_i}\right)} < \frac{\epsilon p}{2e\abs{D_2}}.
		\]  Thus, if $S$ only contains vertices of degree $2$,
		\begin{align*}
			X_S &\leq \frac{2^{ns}}{(2^n)_s} \left(\frac{\epsilon p}{2e\abs{D_2}}\right)^{\abs{S}}
		\end{align*} We then wish to bound the first term in the product:
		\begin{align*}
			 \frac{2^{ns}}{(2^n)_s} &= \prod_{i = 0}^{s - 1} \frac{1}{1 - \frac{i}{2^n}}\\
			 &\leq \left(1 - \frac{s}{2^n}\right)^{-s}\\
			 &\leq \exp\left(\frac{s^2}{2^n}\right)\\
			 &\leq e
		\end{align*} as $\abs{S} \leq \abs{R} \leq  2^{n/2}$. Therefore,
		\begin{align}
				\abs{\bhom_\psi(G,BC(f)) - p^{\abs{D_2}}q^{\abs{D_1}}} &\leq \sum_{\emptyset \neq S\subseteq R(G)} \left(\prod_{r\in R(G)\setminus S}  p_r\right) X_S \nonumber\\
				&\leq e\sum_{\emptyset \neq S\subseteq R(G)} \left(\prod_{r\in R(G)\setminus S}  p_r\right)[ S \cap D_1 = \emptyset]\left(\frac{\epsilon p}{2e\abs{D_2}}\right)^{\abs{S}} \nonumber\\
				&= ep^{\abs{D_2}}q^{\abs{D_1}} \sum_{\emptyset \subsetneq S \subseteq D_2} \left(\frac{\epsilon}{2e\abs{D_2}}\right)^{\abs{S}} \nonumber\\
				&= ep^{\abs{D_2}}q^{\abs{D_1}} \left(\left(1 + \frac{\epsilon}{2e\abs{D_2}}\right)^{\abs{D_2}} - 1\right)\nonumber \\
				&\leq ep^{\abs{D_2}}q^{\abs{D_1}} \left(e^{\frac{\epsilon}{2e}} - 1\right)\label{eqn use bernoulli}\\
				&\leq p^{\abs{D_2}}q^{\abs{D_1}}\epsilon \label{eqn use reverse bernoulli}
		\end{align}
		where we use $1 + x \leq e^x$ in line (\ref{eqn use bernoulli}) and $e^{-x/2} \leq 1 - x$ for $x \leq \half$ in line (\ref{eqn use reverse bernoulli}).	Hence, $f$ also satisfies the Degree Two Homomorphisms Property $DTH(d,\epsilon)$. 	 
	\end{proof}

	To prove our next theorem, we need a few additional definitions. The \textbf{subdivision} of a graph $G$, denoted $\Subdiv(G)$, is the bipartite graph $(A\sqcup B,E)$ where $A = V(G)$, $B = \{r_{(u,v)}: (u,v)\in E(G)\}$, and $E = E_1\sqcup E_2$ where $E_1 = \{(u,r_{(u,v)}): (u,v)\in E(G)\}$,$E_2 = \{(v,r_{(u,v)}): (u,v)\in E(G)\}$. 
	
	For two graphs $H$ and $G$, let $H \leq G$ denote the relation that $H$ is a subgraph of $G$. For a subgraph $H\leq \Subdiv(G)$, let $D_2(H) \subseteq R(H)$ denote the set of vertices in $R(\Subdiv(G))$ of degree $2$ in $G$. Similarly, let $D_A(H)\subseteq R(H)$ and $D_B(H)\subseteq R(H)$ denote the sets of vertices of degree $1$ in $H$ whose incident edge is in $E_1$ and $E_2$ respectively. 

	We will need the following technical Lemma. 
	\begin{lem}\label{lem subgraph expansion}
		 Then,
		\[
			\sum_{\substack{H\\ H\leq \Subdiv(G)}} x^{\abs{D_2(H)}}y^{\abs{D_A(H)}}z^{\abs{D_B(h)}} = (1 + x + y + z)^{\abs{R(\Subdiv(G))}}
		\]
	\end{lem}
	\begin{proof}
		We expand the RHS by the multinomial theorem as
		\[
			(1 + x + y + z)^{\abs{R(G')}} = \sum_{\substack{S,T,U,V\subseteq R(G')\\ S\sqcup T\sqcup U \sqcup V = E(G)}} x^{\abs{S}}y^{\abs{T}}z^{\abs{U}}.
		\] We claim that each term in the sum defines a unique subgraph $H$ of $G'$. Indeed, for a fixed partition of $R(G')$ into sets $S,T,U,V$, let $S$ be the set of vertices in $R(H)$ of degree $2$, let $T$ to be the set of vertices in $R(H)$ of degree $1$ incident to $A$,  let $U$ to be the set of vertices in $R(H)$ of degree $1$ incident to $B$, and let $V$ to all remaining vertices. This mapping is bijective, and thus, 
		\[
			\sum_{\substack{S,T,U,V\subseteq R(G')\\ S\sqcup T\sqcup U \sqcup V = E(G)}} x^{\abs{S}}y^{\abs{T}}z^{\abs{U}} = \sum_{H\leq G'} x^{\abs{D_2(H)}}y^{\abs{D_A(H)}}z^{\abs{D_B(H)}}.
		\] 
	\end{proof}

	\begin{thm}\label{thm degree two embeddings to rainbow embeddings}
		For any fixed $d\geq 1$ and $\epsilon > 0$, the Degree-$2$ Homomorphisms Property $DTH\left(d,\epsilon\left(\frac{5}{2} + \fourier{f}{0}\right)^{-\abs{E(G)}}\right)$ implies the Rainbow Embeddings Property $RAIN(d,\epsilon)$.
	\end{thm}
	\begin{proof}
		Let $G$ be a fixed graph. Suppose $f;\F_2^n\to \spm$ satisfies the Degree-$2$ Homomorphisms Property $DTH\left(d,\epsilon\left(\frac{5}{2} + \fourier{f}{0}\right)^{-\abs{E(G)}}\right)$. Let $\phi:V(G) \hookrightarrow B_d(0)$ be an injection of diameter at most $d$. Recalling the definition of a rainbow embedding, we have
		\begin{align}
			\emb_{\phi}(G,RHG(d,f)) &= \exv_{\chi:E(G)\hookrightarrow \F_2^n} \prod_{(u,v)\in E(G)} [(u,v,\chi((u,v))) \in E(RHG(d,f))]\nonumber\\
			&= \exv_{\chi:E(G)\hookrightarrow \F_2^n} \prod_{(u,v)\in E(G)} [f(\phi(u) + \chi((u,v))) = f(\phi(v) + \chi((u,v))) ]\nonumber\\
		\intertext{Let $U^{+}(u,v)$ denote the event that $f(\phi(u) + \chi((u,v))) = 1$ and  $U^{-}(u,v)$ denote the event that $f(\phi(u) + \chi((u,v))) = -1$. Define $V^+(u,v)$ and $V^-(u,v)$ likewise.}	
				\emb_{\phi}(G,RHG(d,f)) &= \exv_{\chi:E(G)\hookrightarrow \F_2^n} \prod_{(u,v)\in E(G)} \bigg([U^+(u,v)][V^+(u,v)] + [U^-(u,v)][V^-(u,v)]\bigg)\nonumber\\
			&= \exv_{\chi:E(G)\hookrightarrow \F_2^n} \prod_{(u,v)\in E(G)} \bigg([U^-(u,v)][V^-(u,v)] + (1 - [U^-(u,v)])(1 - [V^-(u,v)])\bigg) \nonumber\\
			&= \exv_{\chi:E(G)\hookrightarrow \F_2^n} \prod_{(u,v)\in E(G)} \bigg(1 - [U^-(u,v)] - [V^-(u,v)] + 2[U^-(u,v)][V^-(u,v)]\bigg)\label{eqn prior to lifting}
		\end{align}
		We wish to lift the expression in line (\ref{eqn prior to lifting}) to an equivalent expression in terms of a bipartite graph homomorphism of $\Subdiv(G)$ into $BC(f)$. Observe that $\phi$ is injective and its image has diameter at most $k$, so $\phi$ is a valid choice of the map for the left part of a bipartite graph homomorphism of $\Subdiv(G)$. Similarly, $\chi$ defines the right part of a bipartite graph homomorphism. In consequence, the event $U^-(u,v)$, which holds when $f(\phi(u) + \chi((u,v))) = -1$, is precisely the condition that the edge $(u,r_{(u,v)})\in E(\Subdiv(G))$ is sent to an edge $(\phi(u),\chi(u,v))\in E(BC(f))$.  Similar statements hold for events $U^-(u,v),V^+(u,v),V^-(u,v)$.

		Therefore, the above expression (\ref{eqn prior to lifting}) is precisely the following:
		\begin{align*}
			\emb_{\phi}(G,RHG(d,f)) &=  \exv_{\chi:E(G)\hookrightarrow \F_2^n} \prod_{r_{(u,v)}\in R(\Subdiv(G))} \bigg(1 - [U^-(u,v)] - [V^-(u,v)] + 2[U^-(u,v)][V^-(u,v)]\bigg)
		\end{align*} as each edge in $G$ now corresponds to a unique right vertex in $\Subdiv(G)$. By Lemma \ref{lem subgraph expansion}, we can expand the product as follows: 
		\begin{align}
				\emb_{\phi}(G,RHG(d,f))  &=  \exv_{\chi:E(G)\hookrightarrow \F_2^n} \sum_{H: H\leq \Subdiv(G)} 2^{\abs{D_2(H)}}(-1)^{\abs{D_1(H)}}\prod_{(l,r)\in H} [\phi(l) \underset{BC(f)}{\sim} \chi(r)]\nonumber\\
				&= \sum_{H: H\leq \Subdiv(G)} 2^{\abs{D_2(H)}}(-1)^{\abs{D_1(H)}}\exv_{\chi:E(G)\hookrightarrow \F_2^n} \prod_{(l,r)\in H} [\phi(l) \underset{BC(f)}{\sim} \chi(r)]\label{eqn embeddings to sum}
		\end{align} 
		Define 
		\[
			X = \abs{ \sum_{H: H\leq \Subdiv(G)} 2^{\abs{D_2(H)}}(-1)^{\abs{D_1(H)}}\left(\left(\exv_{\chi:E(G)\hookrightarrow \F_2^n} \prod_{(l,r)\in H} [\phi(l) \underset{BC(f)}{\sim} \chi(r)]\right) - p^{\abs{D_2(H)}}q^{\abs{D_1(H)}}\right)}
		\]
		Now will simplify $X$ in two ways. By Equation (\ref{eqn embeddings to sum}), we have
		\begin{align*}
			X  &=  \abs{\exv_{\chi:E(G)\hookrightarrow \F_2^n} \sum_{H: H\leq \Subdiv(G)} 2^{\abs{D_2(H)}}(-1)^{\abs{D_1(H)}}\left(\left(\prod_{(l,r)\in H} [\phi(l) \underset{BC(f)}{\sim} \chi(r)]\right) - p^{\abs{D_2(H)}}q^{\abs{D_1(H)}}\right)}\\
			&= \abs{\emb_{\phi}(G,RHG(d,f)) -  \sum_{H: H\leq \Subdiv(G)}  (2p)^{\abs{D_2(H)}}(-q)^{\abs{D_1(H)}}}\\
			&= \abs{\emb_{\phi}(G,RHG(d,f)) -  \sum_{H: H\leq \Subdiv(G)}  \left(\half - \fourier{f}{0}\right)^{\abs{D_2(H)}}\left(-\half + \frac{\fourier{f}{0}}{2}\right)^{\abs{D_1(H)}}}\\
		\end{align*}
		We are now in a situation to apply Lemma \ref{lem subgraph expansion} with $x = \half - \fourier{f}{0}$ and $y = z = -\half + \frac{\fourier{f}{0}}{2}$.
		\begin{align*}	
			&= 	\abs{\emb_{\phi}(G,RHG(d,f)) - \prod_{r_{(u,v)}\in R(\Subdiv(G))} \left(1 + 2(-\half + \frac{\fourier{f}{0}}{2}) + \half - \fourier{f}{0}\right)}\\
			&= 	\abs{\emb_{\phi}(G,RHG(d,f)) - 2^{-\abs{E(G)}}}
		\end{align*} where we use the fact that $\abs{R(\Subdiv(G))} = \abs{E(G)}$. On the other hand, the triangle inequality gives us that
		\begin{align}
			X &\leq  \sum_{H: H\leq \Subdiv(G)} 2^{\abs{D_2(H)}}\abs{\left(\left(\exv_{\chi:E(G)\hookrightarrow \F_2^n} \prod_{(l,r)\in H} [\phi(l) \underset{BC(f)}{\sim} \chi(r)]\right) - p^{\abs{D_2(H)}}q^{\abs{D_1(H)}}\right)}\nonumber\\
			&\leq \epsilon(5/2 + \fourier{f}{0})^{-\abs{E(G)}}\sum_{H: H\leq \Subdiv(G)} (2p)^{\abs{D_2(H)}}q^{\abs{D_1(H)}}\label{eqn use degree 2 homs}\\
			&= \epsilon(5/2 + \fourier{f}{0})^{-\abs{E(G)}} \left(1 + 2q + 2p\right)^{\abs{E(G)}} \label{eqn subgraph sum to product}\\
			&\leq \epsilon \nonumber
		\end{align} where we use the fact that $f$ satisfies the Degree-$2$-Homomorphisms Property of rank $d$ with error bound $\epsilon\left(\frac{5}{2} + \fourier{f}{0}\right)^{-\abs{E(G)}}$ in line (\ref{eqn use degree 2 homs}) and Lemma \ref{lem subgraph expansion} in line (\ref{eqn subgraph sum to product}). Putting these two expansions of $X$ together, we find that
		\[
			\abs{\emb_{\phi}(G,RHG(d,f)) - 2^{-\abs{E(G)}}} < \epsilon 
		\] Thus $f$ also satisfies the Rainbow Embeddings Property $RAIN(d,\epsilon)$.
	\end{proof}

	\begin{thm}\label{thm rainbow embeddings to influences}
		For any fixed $d \geq 1$ and $\epsilon > 0$, the Rainbow Embeddings Property $RAIN(d,\epsilon)$ implies the Balanced Influences Property $INF(d,\epsilon)$.
	\end{thm}
	\begin{proof}
		Suppose $f:\F_2^n\to \spm$ satisfies the Rainbow Embeddings Property $RAIN(d,\epsilon)$. Fix $u\in B_d(n,0)$, and let $\phi$ be an injection from $V(K_2)$ to $\{u,0\}$. By definition of rainbow embeddings, we have
		\begin{align*}
			\emb_{\phi}(K_2,RHG(d,f))  &= \exv_{\chi: E(K_2) \to \F_2^n} [(u,0,\chi(e)) \in E(RHG(d,f))]\\
			\intertext{Setting $x = \chi(e)$ and applying the definition of the edge set of $RHG(f)$}
			 \emb_{\phi}(K_2,RHG(d,f)) &= \exv_{x\in \F_2^n} [f(u + x) = f(x)]\\
			 &= \prob_{x\in\F_2^n}\left[f(x + u) = f(x)\right]\\
			 &= 1 - \Inf_{u}[f]
		\end{align*} By Property $RAIN(d,\epsilon)$, we have that ${\emb_{\phi}(K_2,RHG(d,f)) - \half} < \epsilon$. Hence, it follows that $\abs{\Inf_{u}[f] - \half} < \epsilon$ and $f$ also satisfies the Balanced Influences Property $INF(d,\epsilon)$.
	\end{proof}

	\section{Constructions of Quasirandom Functions}\label{sec constructions}
	In this section, we construct a large class of functions which separate the Balanced Influences Property $INF(d + 1,\epsilon)$ from $INF(d,\epsilon')$.

		A \textit{$[n,k,d]$-binary linear code} is a subspace $\mathcal{C}\subseteq \F_2^n$ of dimension $k$ such that $\min_{x \neq y} \abs{x - y} = d$. An $[n,k,d]$ binary linear code may be specified by its \textit{parity check matrix} $M\in \F_2^{(n - k) \times n}$ which has the property that $x\in \mathcal{C} \iff Mx = 0$. Note that the parity check matrix has rank $n - k$. 
		We will need the following elementary fact regarding linear codes of distance $d$.
		\begin{lem}\label{lem property of linear codes}[\cite{Guruswami2019}, Proposition 2.3.5]
			If $M$ is the parity check matrix of a code with distance strictly greater than $d$, then any nonzero $x\in \ker(M)$ must have $\abs{x} > d$.
		\end{lem}
		\begin{ex}\label{ex Hamming code}
			Let $\mathcal{C}$ be the $[7,4]$-Hamming code with parity check matrix $H$
			\[
			H = \begin{bmatrix}
			0 & 1 & 1 & 1 & 1 & 0 & 0 & 0\\
			1 & 0 & 1 & 1 & 0 & 1 & 0 & 0\\
			1 & 1 & 0 & 1 & 0 & 0 & 1 & 0\\
			1 & 1 & 1 & 0 & 0 & 0 & 0 & 1
			\end{bmatrix}
			\] One can check that no vector of Hamming weight $3$ or less can be an element of the kernel, as every set of $3$ columns has at least one row with an odd number of $1$'s
		\end{ex}
			
		The goal of this section is to demonstrate that a bent function composed with the parity check matrix of a distance $d$ linear code is a quasi-random of rank $d$ with error $0$.

		\begin{proof}[Proof of Theorem (\ref{thm tower theorem})]
			Let $\mathcal{C}$ be an $[n,k,d]$ binary linear code such that $n - k$ is even and $n \geq k + 2$. Let $H\in \F_2^{(n - k) \times n}$ be a parity check matrix for $\mathcal{C}$. Let $g:\F_2^{n - k}\to \spm$ be a bent function, and define$f:\F_2^n\to\spm$ by
			\[
				f(x) := g(Hx).
			\] We claim that 
			\[
			\Inf_{\gamma}[f] = \begin{cases}
			\half & \gamma \notin \ker(H)\\
			0 & \gamma\in \ker(H)
			\end{cases}
			\]
			Indeed, by Lemma (\ref{lem influence is autocorrellation is average value of derivative}), we have
			\begin{align}
				\Inf_{\gamma}[f] &= \half - \half f*f(\gamma) \nonumber\\
				&= \half - \half\exv_{\delta\in \F_2^n} g(H\delta)g(H(\delta + \gamma)) \nonumber\\
				&= \half - \half \frac{2^{\rank(H)}}{2^n}\exv_{\eta\in \Range(H)} 2^{\dim(\ker(H))}g(\eta)g(\eta + H\gamma) \nonumber\\
				&= \half - \half\exv_{\eta\in \Range(H)} g(\eta)g(\eta + H\gamma) \label{eqn rank-nullity}
				\end{align}where we use the Rank-Nullity Theorem in line (\ref{eqn rank-nullity}).
				As the parity check matrix is a surjective linear map from $\F_2^n\to \F_2^{n - k}$, we have
				\begin{align}	
				\Inf_{\gamma}[f] &= \half - \half\exv_{\eta\in \F_2^{n - k}} g(\eta)g(\eta + H\gamma) \nonumber\\
				&= \begin{cases}
				\half & \gamma\notin \ker(H)\\
				0 & \gamma\in \ker(H)
				\end{cases}\label{eqn bent convolution}
			\end{align}
			 where we use Lemma (\ref{lem convolution of bent functions}) in line (\ref{eqn bent convolution}). Now we can apply Lemma (\ref{lem property of linear codes}) to conclude that if $\abs{\gamma} \leq d$, $\gamma\notin \ker(H)$. It follows that $\Inf_{\gamma}[f] = \half$ for every $\gamma\in \F_2^n$ with $0 < \abs{\gamma} \leq d$.

			All that remains is to show that $\abs{\fourier{f}{0}} < \half$. To that end we observe 
			\[
				\fourier{f}{0} = \exv_{x\in \F_2^n} g(Hx) = \frac{2^{\dim(\ker(H))}2^{n - k}}{2^n}\exv_{y\in \F_2^{n - k}} g(y) 
			\] by the same reasoning as in line (\ref{eqn rank-nullity}) above. Since $g$ is bent, it follows that
			\[
				\abs{\fourier{f}{0}} = \abs{\fourier{g}{0}} = 2^{-\frac{n - k}{2}}.	
			\] As $n \geq k + 2$, we conclude that $\abs{\fourier{f}{0}} \leq \half$ and $INF(d,0)$ holds for $f$. 
			
			Similarly, as $\mathcal{C}$ has distance $d + 1$, there is some $\gamma'\in\F_2^n$ with Hamming weight $d + 1$ such that $H\gamma' = 0$. Hence, $\Inf_{\gamma'}[f] = 0$ by equation (\ref{eqn bent convolution}) above. Thus $INF(d + 1,\epsilon)$ cannot hold for $f$ unless $\epsilon \geq \half$. 
		\end{proof}
		\begin{rmk}\label{ex IP and Hamming}
			For $r$ even, let $\mathcal{C}$ be the Hamming code $[2^r,2^r - r - 1,3]$, and let $H\in \F_2^{r\times 2^r - 1}$ be its parity check matrix as in Example \ref{ex Hamming code}. Let $IP:\F_2^r\to \spm$ be the inner product function defined in Example (\ref{ex inner product fn}) and define $f(x) := IP(Hx)$. Then, $f$ has the property that $f*f(x) = \begin{cases}
			-1 & x\in \mathcal{C}\\
			1 & x\notin \mathcal{C}
			\end{cases}$  
		\end{rmk}
	\section{Proofs of relations among extant Quasirandomness Theories for Boolean Functions}\label{sec comparisons}
		 We will prove a series of lemmas which together separate and relate the classes of quasirandom boolean functions defined in section \secref{sec previous theories}. 
		\begin{lem}\label{lem balanced influences s to F2 degree 1 regular}
			If $f:\F_2^n\to \spm$ has the Balanced Influences property $INF(d,\epsilon/2)$, then $f$ is $(\sqrt{2^{-d} + \epsilon},n)$ $\R$-Regular. 
		\end{lem}
		\begin{proof}
			By Theorem \ref{lem influences to spectral sample}, if $f$ has the Balanced Influences Property $INF(d,\epsilon/2)$, then $f$ has the Spectral Discrepancy Property $SD(d,\epsilon)$. Fix $\gamma\in \F_2^n$ and let $C(S,z)$ be a subcube of dimension $n - d$ which contains $\gamma$. By $SD(d,\epsilon)$,
			\[
				\fourier{f}{\gamma}^2 \leq \sum_{\delta\in C(S,z)} \fourier{f}{\delta}^2 \leq 2^{-d} + \epsilon.
			\] Therefore, $\abs{\fourier{f}{\gamma}} \leq \sqrt{2^{-d} + \epsilon}$ for every $\gamma\in \F_2^n$. 
		\end{proof}
		
		\begin{lem}\label{lem gownorm 3 greater than balanced influences}
			For even $n$, there is a function $f:\F_2^{n}\to \spm$ which has $INF(d,0)$ for every $d \leq n$, but $\gownorm{f}{3} = 1$. 
		\end{lem}
		\begin{proof}
			Consider the Inner Product function $IP(x):\F_2^n\to \spm$ defined in Example \ref{ex inner product fn}. As shown in the example, $IP$ is a bent function and therefore has the property $INF(d,2^{-n/2})$ for every $1 \leq d \leq n$. However, $IP$ has $\F_2$-Degree $2$. Since $\gownorm{g}{d + 1} = 1$ if $g$ has $\F_2$-degree $d$ \cite{Hatami2019}, we conclude that $\gownorm{IP}{3} = 1$. 
		\end{proof}

		\begin{lem}\label{lem balanced influences greater than gownorm}
				Let $g:\F_2^n\to \spm$ be a boolean function. Let $M\in \F_2^{(n + 1)\times n}$ be the projection matrix which sends $x\in \F_2^{n + 1}$ to its first $n$ coordinates, and let $w\in \F_2^{n + 1}$ be the vector with a single $1$ in the $n + 1$st coordinate. Let $f:\F_2^{n + 1}\to \spm$ be defined by $f(x) = g(Mx)$. If $g$ is $(\epsilon,k)$ $\F_2$-Regular, then 
				\begin{itemize}
					\item $f$ is $(\epsilon,k)$ $\F_2$-Regular
					\item $\Inf_{w}[f] = 0$.
				\end{itemize}
		\end{lem}
		\begin{proof}
			 We first show that $f$ is $(\epsilon,k)$ $\F_2$-Regular. To that end, we have
			\begin{align}
				\gownorm{f}{k} &= \left(\exv_{x\in \F_2^n} \exv_{N\in \F_2^{n\times k}} \prod_{v\in \F_2^k} f(x + Nv)\right)^{2^{-k}}\nonumber\\
				&= \left(\exv_{x\in \F_2^n} \exv_{N\in \F_2^{n\times k}} \prod_{v\in \F_2^k} g(M(x + Nv))\right)^{2^{-k}}\nonumber\\
				&= \left(\exv_{x\in \F_2^n} \exv_{N\in \F_2^{n\times k}} \prod_{v\in \F_2^k} g(Mx + MNv)\right)^{2^{-k}}\nonumber\\
				&= \left(\frac{2^{k + 1}2^{(n - 1)(k + 1)}}{2^{n(k + 1)}}\exv_{y\in \F_2^{n - 1}} \exv_{P\in \F_2^{n - 1\times k}} \prod_{v\in \F_2^k} g(y + Pv)\right)^{2^{-k}}\label{eqn projection}\\
				&= \gownorm{g}{k}\label{eqn use gowers norm bound}\\
				&\leq \epsilon
			\end{align} where we use the fact that $M$ is a projection in line (\ref{eqn projection}), and use our assumption on $g$ in line (\ref{eqn use gowers norm bound}). Thus $f$ is $(\epsilon,k)$ $\F_2$-Regular.
			
			For the second claim, we observe that $f(x + w) = f(x)$ for every $x$. Therefore, $\Inf_{w}[f] = 0$. It follows that $f$ cannot have $INF(d,\epsilon)$ for any $d\geq 1$ and $\epsilon < \half$. 
		\end{proof}

		Now we can prove each of theorems relating our properties to extant theories. 
		\begin{proof}[Proof of Theorem (\ref{thm relation to F2 regularity})]
			We have three claims to prove. First, we consider the relationship between $(\epsilon,n)$ Balanced Influences and $(\epsilon,1)$ $\F_2$-Regularity. Let $f:\F_2^n\to \spm$ satisfy $INF(d,\epsilon)$. By Lemma \ref{lem balanced influences greater than gownorm}, $f$ is $(\sqrt{2^{-d} + \epsilon},n)$ $\R$-regular. By Proposition 6.7 in O'Donnell's book \cite{o'donnell_2014}, $(\epsilon,1)$ $\F_2$-Regularity is equivalent to $(\epsilon,n)$ $\R$-Regularity. Thus, $f$ is $(\sqrt{2^{-d} + \epsilon},1)$ $\F_2$-Regular. 
			
			Now we show that $(\epsilon,k)$ Balanced Influences is comparable with $(\epsilon,d)$ $\F_2$-Regularity for $d > 1$ and any $k$. First we show that Balanced Influences $INF(d,\epsilon)$ cannot imply $(\epsilon,2)$ $\F_2$-Regularity.   Lemma \ref{lem gownorm 3 greater than balanced influences} provides a bent function $f$ such that $\gownorm{f}{3} = 1$. It follows that the class of bent functions is distinct from $(\epsilon,d)$ $\F_2$-Regular functions for any $d \geq 2$ and $\epsilon < 1$. Since every Bent function satisfies $INF(k,0)$ for any $k\in \N$ with $1 \leq k \leq n$, it follows that $INF(d,\epsilon)$ cannot imply $(\epsilon,2)$ $\F_2$-Regularity.

			Now we can show that $(\epsilon,d)$ $\F_2$-Regularity cannot imply $(\epsilon,k)$ Balanced Influences for any $k \geq 1$. Let $g:\F_2^n\to \spm$ be quadratic residue function considered in Example 4.12 of \cite{castro-silva2021quasirandomness}. For any $\epsilon > 0$, there is $n$ sufficiently large such that $g$ is $(\epsilon,k)$ $\F_2$-Regular.  By lemma \ref{lem balanced influences s to F2 degree 1 regular} applied to $g$, we find an $f:\F_2^{n + 1}\to \spm$ which is $(\epsilon,k)$ $\F_2$-Regular yet there is a vector $w\in \F_2^{n + 1}$ such that $\Inf_{w}[f] = 0$.  Thus, $f$ cannot have the Balanced Influences Property $INF(k,\epsilon)$ for any $k \geq 1$ and $\epsilon < 1$.
			
			 It follows that $(\epsilon,k)$ $\F_2$-Regularity and Quasirandomness of rank $d$ with error $\epsilon$ are incomparable for $k \geq 2$ and $d \geq 1$. 
		\end{proof}
	
		\begin{proof}[Proof of Theorem \ref{thm relation to R-regularity}]
			We have two claims to show. First, we prove that $(\delta,k)$ Balanced Influences implies $(\epsilon,n)$ $\R$-Regularity. Lemma \ref{lem balanced influences s to F2 degree 1 regular} implies that if $f:\F_2^n\to \spm$ has $INF(d,\epsilon)$, then $f$ is also $(\sqrt{2^{-d} + \epsilon},n)$ $\R$-Regular. Since $(\epsilon,k)$ $\R$-Regularity implies $(\epsilon,k - 1)$ $\R$-Regularity by definition, it follows that any $f:\F_2^n\to \spm$ with $(\epsilon,d)$ Balanced Influences also satisfies $(\epsilon,k)$ $\R$-Regularity for any $k \leq n$.

			For the second claim, we must show that $(\epsilon,k)$ $\R$-Regularity cannot imply $(\epsilon,d)$ Balanced Influences for any $k \leq n$, $d \geq 1$ or $\epsilon < 1$. Consider the inner product function $IP:\F_2^{2n}\to \spm$ defined in Example \ref{eqn IP proof orthogonality}. By applying Lemma \ref{lem balanced influences greater than gownorm} to $IP$, we find an $f:\F_2^{2n + 1}\to \spm$ which is  $(2^{-n/2},n)$ $\R$-regular and yet does not have $INF(d,\epsilon)$ for any $d \geq 1$ and $\epsilon < \half$. As $(\epsilon,n)$ $\R$-regularity implies $(\epsilon,k)$ $\R$-regularity for $k < n$, it follows that $(\epsilon,k)$ $\R$-regularity does not imply $INF(d,\epsilon)$ for any $k,d,\epsilon$. Thus $(\delta,k)$ $\R$-Regularity cannot imply $(\epsilon,d)$ Balanced Influences for any $\delta > 0$ or $k \leq n$.
		\end{proof}
	
		\begin{proof}[Proof of Theorem (\ref{thm relation to small stable influences})]
			 Assume $f$ satisfies $SD(d,\epsilon e^{(d - 1)((\delta/2) - \ln(2))})$. We want show that 
			 \[
			 	\Inf_{i}^{1 - \delta}[f] = \sum_{\substack{\gamma\in \F_2^n\\ \gamma_i = 1}} (1 - \delta)^{\abs{\gamma} - 1}\fourier{f}{\gamma}^2
			 \] is at most $\epsilon$. We observe that the set of $\gamma\in \F_2^n$ with $\gamma_i = 1$ is precisely the $n - 1$-dimensional subcubes $C(\{i\},1)$, and the same subcube may be divided into $2^{d - 1}$ subcubes of dimension $n - d$ follows. Pick a set $S$ of size $d$ such that $\{i\} \subset S \subset [n]$. Then, $C(\overline{\{i\}},1) = \bigsqcup_{\substack{z\in \F_2^S\\ z_i = 1}} C(S,z)$. Therefore,
			 \begin{align*}
			 		\Inf_{i}^{1 - \delta}[f] &= \sum_{\substack{\gamma\in \F_2^n\\ \gamma_i = 1}} (1 - \delta)^{\abs{\gamma} - 1}\fourier{f}{\gamma}^2\\
			 		&= \sum_{\substack{z\in \F_2^S\\ z_i = 1}} \sum_{\gamma\in C(S,z)} (1 - \delta)^{\abs{\gamma} - 1}\fourier{f}{\gamma}^2\\
			 		&\leq \sum_{\substack{z\in \F_2^S\\ z_i = 1}} \left(\max_{\gamma\in C(S,z)} (1 - \delta)^{\abs{\gamma} - 1}\right)\sum_{\gamma\in C(S,z)} \fourier{f}{\gamma}^2\\
			 		&\leq \epsilon e^{(d - 1)((\delta/2) - \ln(2))}\sum_{\substack{z\in \F_2^S\\ z_i = 1}} \left(\max_{\gamma\in C(S,z)} (1 - \delta)^{\abs{\gamma} - 1}\right)\\
			 \end{align*}
			 where we use $SD(d,\epsilon e^{(d - 1)((\delta/2) - \ln(2))})$ in the ultimate line. Now we can simplify the result further:
			 \begin{align*}		
			 		\Inf_{i}^{1 - \delta}[f] &\leq  \epsilon e^{(d - 1)((\delta/2) - \ln(2))}\sum_{\substack{z\in \F_2^S\\ z_i = 1}}  (1 - \delta)^{\abs{z} - 1}\\
			 		&= \epsilon e^{(d - 1)((\delta/2) - \ln(2))}\left(\sum_{j = 0}^{d - 1} \binom{d - 1}{j}(1- \delta)^j\right)\\
			 		&= \epsilon e^{(d - 1)((\delta/2) - \ln(2))}(2 - \delta)^{d - 1}\\
			 		&= \epsilon e^{(d - 1)(\delta/2)}\left(1 - \frac{\delta}{2}\right)^{d - 1}\\
			 		&\leq \epsilon
			 \end{align*}
			 where we use $(1 + x) \leq e^{x}$ in the final line. Thus $f$ has $(\epsilon,\delta)$-Small Stable Influences. 
			 
			 Conversely, one can easily verify that $\chi_{\bone}$ has $((1 - \delta)^n,\delta)$-Small Stable Influences, but $\Inf_{\gamma}[\chi_{\bone}] =1$ for every $\gamma\in \F_2^n$ with Hamming weight $1$. Thus $\chi_{\bone}$ does not have $INF(d,\epsilon)$ for any $d \geq 1$ unless $\epsilon = \half$.  
		\end{proof}
	
		The relationship between $\Zp{2^n}$ Regularity and the other theories is more intricate than our other theories of quasi-randomness, largely due to the algebraic differences between $\Zp{2^n}$ and $\F_2^n$. As boolean functions in the sense of $\Zp{2^n}$-Regularity are not functions on $\F_2^n$, we have the following definition to transfer results between these two theories:
		\begin{defn}\label{defn transfer F_2^n to Z/2^nZ}
			Given $z\in \Zp{2^n}$, let  $z^*\in \F_2^n$ denote the vector such that
			\[
				z_i^* = a_i
			\] where $z = \sum_{i = 1}^{n} a_i2^{i - 1}$ is the binary expansion of $z$.
		\end{defn}

		Chung and Graham [\cite{Chung1992}, Prop. 6.2] prove the following result, translated into our notation:
		
		\begin{lem}\cite{Chung1992}\label{lem all ones vector is Zp pseudorandom}
			 There is an absolute constant $C$ such that the function $z\to \chi_{\bone}(z^*)$ is $\Zp{2^n}$-Regular with error bound $C\left(\dfrac{\sqrt{2 + \sqrt{2}}}{2}\right)^n \approx 0.92^n$.
		\end{lem}
	
		As $\chi_{\bone}$ is a Fourier character, $\chi_{\bone}$ cannot be $(\epsilon,n)$ $\R$-Regular for any $ \epsilon < 1$. Thus for any $\delta > 0$, $\Zp{2^n}$ Regularity with error bound $\delta$ does not imply $(\epsilon,n)$ $\R$-Regularity for any $\epsilon < 1$. Here we extend their result to show that $\Zp{2^n}$ Regularity with error bound $\delta$ cannot even imply $(\epsilon,k)$ $\R$-Regularity for a wide range of $k < n$. 
		
		\begin{proof}[Proof of Theorem \ref{thm relation to Z/2^nZ regularity}]
			Set $k = \lrceil{-C_0\ln(\epsilon)}$ for some absolute constant $C_0$ to be defined later. Define $S=\{1,\dots,n - k\}$. Define $\gamma\in \F_2^n$ by $\gamma:= \bone\underset{S}{\otimes} 0$ where $\bone\in \F_2^S$ is the all-ones vector and $0\in \F_2^{\overline{S}}$ is the zero vector. We will show that $\chi_{\gamma}$ is $\epsilon$ $\Zp{2^n}$-Regular.  
			
			Define $\omega_n := \exp\left(\dfrac{2\pi i }{2^n}\right)$. Now let $c\in \Zp{2^n}\setminus \{0\}$ be arbitrary, and via the Euclidean algorithm, write $c = 2^ka + b$ where $0 \leq b < 2^k$. Then,
			\begin{align*}
				\exv_{z\in \Zp{2^n}} \chi_{\gamma}(z^*)\omega_n^{-cz} &= \exv_{0 \leq y < 2^{n - k}} \exv_{0 \leq x < 2^{k}} \chi_{\gamma}(y^*\underset{S}{\otimes} x^*)\omega_n^{-(2^ka + b)(2^{n - k}x + y)}\\
				\intertext{We have $\chi_{\gamma}(y^*\underset{S}{\otimes} x^*) = \chi_{\bone}(y^*)\chi_{0}(x^*) = \chi_{j}(y^*)$ by definition of $\gamma$. Hence,}
				&= \exv_{0 \leq y < 2^{n - k}} \exv_{0 \leq x < 2^{k}} \chi_{\bone}(y^*)\omega_n^{-2^nax - 2^{n - k}xb - 2^kay - by}\\
				&= \exv_{0 \leq y < 2^{n - k}} \exv_{0 \leq x < 2^{k}} \chi_{\bone}(y^*)\omega_n^{- 2^{n - k}xb - 2^kay - by}\\
				&= \exv_{0 \leq y < 2^{n - k}} \chi_{\bone}(y^*)\omega_n^{-2^kay - by}\exv_{0 \leq x < 2^{k}} \omega_n^{-2^{n - k}xb}\\
				&= \exv_{0 \leq y < 2^{n - k}} \chi_{\bone}(y^*)\omega_n^{-2^kay - by}\exv_{0 \leq x < 2^{k}} \omega_k^{-xb}\\
				&= \begin{cases}
				0 & b \neq 0\\
				\exv_{0 \leq y < 2^{n - k}} \chi_{\bone}(y^*)\omega_n^{-2^kay} & b = 0
				\end{cases}\\
				&= \begin{cases}
				0 & b \neq 0\\
				\exv_{0 \leq y < 2^{n - k}} \chi_{\bone}(y^*)\omega_{n - k}^{ay} & b = 0
				\end{cases}
			\end{align*}
			We may now apply Lemma \ref{lem all ones vector is Zp pseudorandom} so that
			\begin{align*}
				\abs{\exv_{z\in \Zp{2^n}} \chi_{\gamma}^*(z)\omega_n^{-cz} } &\leq C\frac{\sqrt{2 + \sqrt{2}}}{2}^{k}\\
				&\leq C\left(\dfrac{\sqrt{2+\sqrt{2}}}{2}\right)^{-C_0\ln(\epsilon)}\\
				&= \epsilon
			\end{align*}
			where $C_0$ is a sufficiently large absolute constant. 
			Thus $z\to \chi_{\gamma}(z^*)$ is $\epsilon$-$\Zp{2^n}$ Regular. However, $\abs{\gamma} = n - k$, and so $\chi_{\gamma}$ cannot be $(\epsilon,n - k + 1)$ $\R$-Regular for any $\epsilon < 1$. Thus $\delta$-$\Zp{2^n}$ Regularity does not imply $(\epsilon,n - k)$ $\R$-Regularity for any $\epsilon < 1$. 
		\end{proof}

	\section{Problems and remarks}\label{sec conclusion}
		Our proof of Theorem (\ref{thm tower theorem}) provides a large class of examples of functions which satisfy $INF(d,0)$ but not $INF(d + 1,\epsilon)$ for $\epsilon$ small and any $d \geq 1$. Furthermore, these functions have the property that $f*f(x)$ is the indicator function for some binary linear code. Many questions remain, some of which we include here. 
		\begin{ques}\label{ques nonlinear codes}
			Let $f:\F_2^n\to\spm$ be a boolean function such that
			\[
				f*f(x) = [x\in \mathcal{C}]
			\] for some \emph{nonlinear} binary code $\mathcal{C} \subseteq \F_2^n$ of distance $d$. Do such functions exist, and if so, is it true that $f$ has $INF(d,0)$ but not $INF(d + 1,\epsilon)$ for $\epsilon< \half$?
		\end{ques}
		\begin{ques}\label{ques classification}
			Is there a classification of functions $f:\F_2^n\to \spm$ which satisfy 
			\[
			f*f(x) = [x\in \mathcal{C}]
			\] for some $[n,k,d]$ binary linear code $\mathcal{C}\subseteq \F_2^n$?
		\end{ques}
		We remark that any progress on this question will lead towards a solution of the problem of enumerating bent functions. 
		
		There are several other interesting directions to consider.The relationship between our work and $(\epsilon,d)$-$\F_2$-Regularity is summarized in Figure \ref{figure relation with F2 regularity}.
		\begin{figure}[htbp]
			\centering
			\begin{tikzpicture}
			\tikzstyle{statement} = [rectangle,draw=black,fill = white,minimum width= 2cm,minimum height= 1cm];
			\tikzstyle{Implication} = [line width=1mm,draw=black!50];
			\tikzstyle{Incomparable} = [dotted,line width=1mm,draw=red!50];			
			\node[statement] (gow1) at (2,0) {$(\epsilon,1)$ $\F_2$-Regular};
			\node[statement] (gow2) at (2,2) {$(\epsilon,2)$ $\F_2$-Regular};

			\node[statement] (gowk) at (2,7) {$(\epsilon,d)$ $\F_2$-Regular};
			
			\node[blue,statement] (inf1) at (-4,1) {\textbf{$(\epsilon,1)$-Balanced Influences}};
			\node[blue] (infdots) at (-4,3) {$\vdots$};
			
			\node[blue,statement] (infn) at (-4,5) {\textbf{$(\epsilon,n)$-Balanced Influences}};
			\node (bent) at (-4,7) {Bent};
			
			\draw[->,Implication] (gowk) -- node[right] {\cite{castro-silva2021quasirandomness,chung1993communication}} (2,5.5);

			\draw[->,Implication] (2,3.5) -- node[right] {\cite{castro-silva2021quasirandomness,chung1993communication}} (gow2);
			
			\draw[->,Implication] (gow2) -- node[right] {\cite{castro-silva2021quasirandomness,chung1993communication}} (gow1);
			
			\foreach \i in {3.67,4.33,5}{
				\node[circle,minimum size=1mm,inner sep=0pt,outer sep=0pt,fill=black!50] at (2,\i) {};
			}
			
			
			\draw[blue,->,Implication] (infn) -- (gow1) node[midway,pos=0.67,sloped,above] {Theorem \ref*{thm relation to F2 regularity}};
			\draw[blue,->,Implication] (bent) -- (infn) node[midway,left] {Def.};

			\draw[blue,->,Implication] (infn) --  (-4,4) node[midway,left] {Def.};
			\draw[blue,->,Implication] (-4,2) --  (inf1) node[midway,left] {Def.};
			\draw[blue,->,Implication] (inf1) -- node[midway,below,sloped] {Theorem \ref*{thm relation to F2 regularity}} (gow1);
			
			\foreach \i in {2.33,3,3.67}{
				\node[circle,minimum size=1mm,inner sep=0pt,outer sep=0pt,fill=black!50] at (-4,\i) {};
			}
			
			\draw[red,dotted,Incomparable,->] (bent) -- node[pos=0.25,above,sloped] {Theorem \ref*{thm relation to F2 regularity}} node[solid,midway,cross,sloped,minimum width =12pt,minimum height =12pt] {} (gow2);
			\draw[red,dotted,Incomparable,->] (gowk) -- node[pos=0.25,above,sloped] {Theorem \ref*{thm relation to F2 regularity}} node[solid,midway,cross,sloped,minimum width =12pt,minimum height =12pt] {} (inf1);	
			\end{tikzpicture}
			\caption{The relationships between the properties in Theorem \ref{thm main result} and $\F_2$-Regularity. Each arrow is a strict implication. Each arrow which follows by definition has no label, and all other arrows have a label with a reference to the proof of the implication.  The results of this paper are in bold blue text and dashed blue arrows. Incomparable properties are linked by red dotted lines.}
			\label{figure relation with F2 regularity}
		\end{figure}
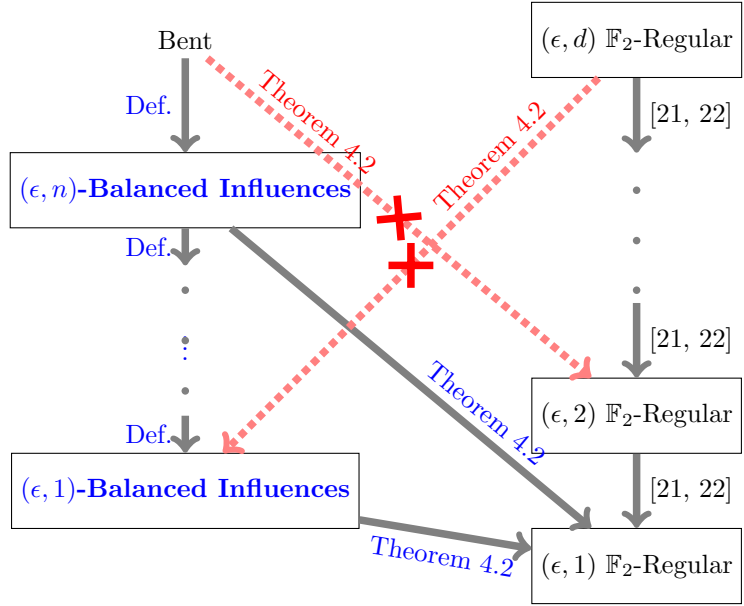
	 One can observe that our quasirandom theorem takes a new direction off $(\epsilon,1)$-$\F_2$-Regularity in the quasirandom hierarchy defined in \cite{castro-silva2021quasirandomness}. One might ask if there is an analogue of these results for the $i$th level of the same hierarchy. As the Balanced Influences Property $INF(n,0)$ is equivalent to a function being bent, any analogue of our results for higher levels may provide a ``higher order'' analogue of bent functions. Such functions are likely to have very strong cryptographic properties and any constructions of them would be interesting in their own right. 
	 
	 Finally, the relationship between quasirandomness over $\F_2^n$ and over $\Zp{2^n}$ seems to be of particular interest. We have the following conjecture, largely from numerical evidence:
	 \begin{conj}\label{conj relation to Z/2^nZ regularity}
	 	For any $\epsilon > 0$, there is a $k = k(\epsilon)$ and a $\delta = \delta(\epsilon) > 0$ such that $(\delta,k)$-$\R$-regularity implies $\epsilon$-$\Zp{2^n}$-Regularity.
	 \end{conj}
 	
	\bibliographystyle{ieeetr}
	\bibliography{quasi_random_bib}	

\end{document}